\documentclass[reqno,12pt]{amsart}
\usepackage[left=1.2in, right = 1.2in, top=1in]{geometry}
\usepackage{amsmath, amssymb}
\usepackage{amsthm}
\usepackage{tikz}
\usepackage{amstext}
\usepackage{pifont}
\usepackage{appendix}
\usepackage{verbatim}
\usepackage{cases}
\usepackage{CJKutf8}
\newtheorem{theorem}{Theorem}[section]
\newtheorem{remark}{Remark}[section]

\newtheorem{lemma}{Lemma}[section]

\newtheorem{claim}{claim}

\usepackage{fancyhdr}
\usepackage{graphicx}
\usepackage{float}
\allowdisplaybreaks[4]

\title{Liouville theorem for elliptic equations with a source reaction term involving\\ the product of the function and its gradient in $\mathbb R^n$}

\author{Xi-Nan Ma}
\address{Department of Mathematics \\University of Science and Technology of China, Hefei, China}
\email{xinan@ustc.edu.cn}
\author{Wangzhe Wu}
\address{Institute of Mathematics \\Academy of Mathematics and Systems Science, Chinese
Academy of Sciences, Beijing, 100190, China}
\email{wuwz18@mail.ustc.edu.cn}

\begin{document}
	\begin{CJK}{UTF8}{gbsn}
	\pagestyle{fancy}

\fancyhead{}                                                          
\fancyhead[CO]{Liouville theorem for a class elliptic equations }

\fancyhead[CE]{\leftmark}

	\begin{abstract}
		We improve the Liouville theorem for the equation $-\Delta v = v^p |\nabla v|^q$ in $\mathbb R^n$, which was studied by  Bidaut-V\'eron,  Garc\'ia-Huidobro,
		and  V\'eron \cite{MR3959864}. The proof is based on a differential identity and Young inequality.
	\end{abstract}

	\maketitle

	\section{Introduction}
	We study the equation
	\begin{equation}\label{intro_equ1}
		-\Delta v = v^p |\nabla v|^q, v\geq 0, \text{ in }\mathbb R^n ,
	\end{equation}
	with $p \geq 0, p + q > 1, 0\leq q < 2$.
	If $q = 0$, by Gidas and Spruck \cite{MR615628}, we know when $1 < p < \frac{n + 2}{n - 2}$, all the nonnegative solutions are zero.  Gidas, Ni and Nirenberg \cite{MR634248}  proved that for $p = \frac{n + 2}{n - 2}$, all the positive solutions with reasonable behavior at infinity, namely $v = O(|x|^{2 - n})$ , are radially symmetric about some point, and have the form
	$$
	u = \Big[\frac{\lambda \sqrt{n(n - 2)}}{\lambda^2 + |x - x_0|^2} \Big]^{\frac{n - 2}{2}}.
	$$
	This uniqueness result, as was pointed out by R. Schoen, is in fact equivalent to the geometric result due to Obata \cite{MR303464}.
	
	Due to Caffarelli, Gidas and Spruck \cite{MR982351}, the treatment of the critical case $p = \frac{n + 2}{n - 2}$ was made possible thanks to a completely new approach based upon a combination of moving plane analysis and geometric measure theory.
	Later, Chen-Li \cite{MR1121147} provided a simpler proof. For more application of the differential identity and integration by parts, one can refer to \cite{MR4753063, MR4519639}. Recently, Catino and Monticelli \cite{catino} get some classification results on Riemannian manifolds. Besides the group of Youde Wang use Nash-Moser iteration to get some gradient estimates and Liouville type theorems for complete Riemannian manifolds \cite{ he2024nashmoseriterationapproachlogarithmic, he2024optimalliouvilletheoremslaneemden, MR4703505, MR4559367}.  
	
	If $p = 0$, this is the Hamilton-Jacobi equation. By P. L. Lions \cite{MR833413}: any $C^2$ solutions in $\mathbb R^n$ have to be a constant for $q > 1$. For the $p$-laplacian case, one can refer to the work of Bidaut-V\'eron, Garcia-Huidobro and V\'eron \cite{MR3261111}. Besides they in \cite{MR4150912} and \cite{MR4043662} study the equation:
	\begin{equation}
		\Delta v + N v^p + M |\nabla v|^{q}=0,
	\end{equation} 
	and some Liouville type results are obtained.
	 In \cite{MR4234084} and  \cite{MR3959864} , they study the equation \eqref{intro_equ1} and obtain the following results:
	\begin{theorem}[Theorem C in \cite{MR3959864} ]\label{8.24G}
		Assume that $p \geq 0, p + q > 1, 0 \leq q < 2$ and define the polynomial $G$ by
		\begin{equation}\label{8.23G}
			\mathbb G(p, q)= \Big[ (n - 1)^2q + n - 2\Big] p^2 + b(q)p - nq^2,
		\end{equation}
		where 
		\begin{equation}
			b(q) = n(n - 1)q^2 - (n^2 +n - 1)q - n - 2.
		\end{equation}
		If the couple $(p, q)$ satisfies the inequality $\mathbb G(p, q) < 0$, then all the positive solutions of \eqref{intro_equ1} in $\mathbb R^n$ are constant.
	\end{theorem}
	Besides, they also get the sufficient and necessary condition for the existence of radial positive solutions:
	\begin{theorem}[Theorem D in \cite{MR3959864}]
		There exist nonconstant radial positive solutions of \eqref{intro_equ1} in $\mathbb R^n$ if and only if $p \geq 0$, $0\leq q < 1$, and 
		\begin{equation}\label{intro_equ2}
			p(n - 2) + q(n - 1) \geq n + \frac{2 - q}{1 - q}.
		\end{equation}
		If equality holds in \eqref{intro_equ2}, then there exists an explicit 1-parameter family of positive radial solutions of \eqref{intro_equ1} in $\mathbb R^n$ under the form 
		\begin{equation}
			v_c(r) = c\Big[ Kc^{\frac{(2 - q)^2}{(n - 2)(1 - q)}} + r^{\frac{2 - q}{1 - q}}\Big]^{-\frac{(n - 2)(1 - q)}{2 - q}},
		\end{equation}
		for any $c > 0$ and some $K = K(n, q) > 0$.
	\end{theorem}
	The natural problem is that when $\mathbb G(p, q)\geq 0$ and
	\begin{equation}
			p(n - 2) + q(n - 1) < n + \frac{2 - q}{1 - q},
	\end{equation}
	whether the equation \eqref{intro_equ1} has nontrivial positive solutions?
	This is the result of this paper, and we answer the question partly. The first result is optimal when $0 \leq q \leq \frac{1}{n - 1}$.
	
	\begin{theorem}
		Assume that $p \geq 0, p + q > 1, 0 \leq q \leq \frac{1}{n - 1}$ and 
		\begin{equation}
			p(n - 2) + q(n - 1) < n + \frac{2 - q}{1 - q},
		\end{equation}
		then all the positive solutions of \eqref{intro_equ1} in $\mathbb R^n$ are constant.
	\end{theorem}
	The second result is better than Theorem \ref{8.24G} .
	\begin{theorem}
		Define
		\begin{equation}
			\mathbb H(p, q):= p^2 + \Big[\frac{n - 1}{n - 2}q - \frac{n^2 - 3}{(n - 2)^2}\Big]p + \frac{1 - (n - 1)q}{(n - 2)^2}.
		\end{equation}
		Assume that $p \geq 0, p + q > 1, \mathbb H(p, q) < 0$ and suppose $ \frac{1}{n - 1} < q < 1$ when $n = 3$ and  $\frac{1}{n - 1} < q < 2$ when $n \geq 4$,
		then all the positive solutions of \eqref{intro_equ1} in $\mathbb R^n$ are constant.
	\end{theorem}
	By direct computation, we know
	\begin{lemma}\label{compare}
		If $p \geq 0, p + q > 1, \frac{1}{n - 1} < q < 2$ and $\mathbb G(p, q) < 0$, then
		\begin{equation*}
			0 > p^2 + \Big[\frac{n - 1}{n - 2}q - \frac{n^2 - 3}{(n - 2)^2}\Big]p + \frac{1 - (n - 1)q}{(n - 2)^2}.
		\end{equation*}
	\end{lemma}
	
	Define: $f1: p + q = 1;$ \quad \quad $f2:	p(n - 2) + q(n - 1) = n + \frac{2 - q}{1 - q};$\\$f3: \mathbb{G}(p, q) = 0;$\quad \quad$f4: \mathbb H(p,q) = 0, q > \frac{1}{n - 1};$ \quad \quad $f5: q = 1.$
	\begin{figure}[H]
		\centering
		\includegraphics[width=15cm, height=8cm]{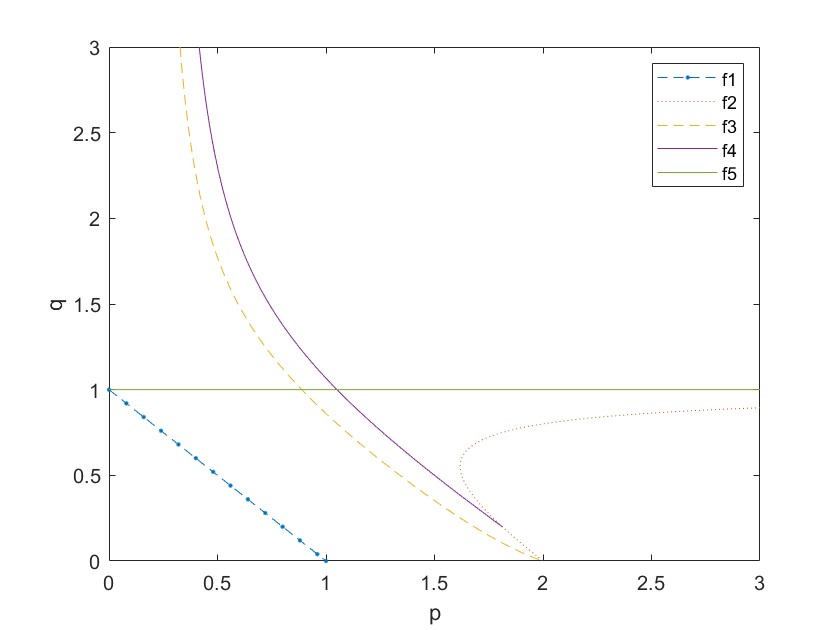}
		\caption{Graphs for $n = 6$.}		
	\end{figure}
	
	 One new idea in this paper is that we introduce the new parameter $S,Q$ in (\ref{1.1a}). Before our work one always let $S=Q=\frac{1}{n}$ as in Gidas and Spruck \cite{MR615628}. This idea obtains one more freedom to get the positive term in (\ref{sec2_equ1}).

	The proof of these two theorems is based on a differential identity and Young inequality. In section \ref{1}, we get the differential identity. In section \ref{conditions}, we deduce the differential inequality under some conditions. In section \ref{Young}, we show how to use Young inequality to prove that the solutions are constant. Finally in section 5, we obtain the conditions on $(p, q)$ and give the proof of the theorems and Lemma \ref{compare}.

		{\it Acknowledgement.} The authors thank Prof. Bidaut-V\'eron 
	and  V\'eron for their very useful comments on this paper.
	The authors was supported by  National Natural Science Foundation of China (grants 12141105) and National Key Research and Development Project (grants SQ2020YFA070080).

	\section{Differential identity}\label{1}
	Consider the equation:
	\begin{equation}\label{section1_equ1}
		\Delta v + v^p|\nabla v|^q = 0.
	\end{equation}
Multiplying \eqref{1} by $v^{\alpha}|\nabla v|^{\gamma}\Delta v$,
\begin{equation}\label{2}
		\begin{aligned}
			   &v^\alpha|\nabla v|^{\gamma}(\Delta v)^2  =  -  v^{\alpha + p}|\nabla v| ^{\gamma + q}\Delta v.
		\end{aligned}
	\end{equation}
	We observed that the term $v^\alpha|\nabla v|^{\gamma}(\Delta v)^2 $ is important. 
	In the following contents, we always suppose that $v_1(x_0) = |\nabla v|(x_0)$ and define
	\begin{equation}\label{1.1a}
		\begin{aligned}
			&G_{11} = v_{11} - S \Delta v,\\
			&G_{ij} = v_{ij} - Q\delta_{ij}\Delta v, i + j > 2.
		\end{aligned}
	\end{equation}	
	We remark that the definition of $G_{ij}$ is depending on $x_0$. Besides our computations in this section are always at one point $x_0$ with $|\nabla v|(x_0) \neq 0$.
	
	\begin{itemize}
		\item The term $v^{\alpha - 1}|\nabla v|^{\gamma + 2}\Delta v $:

	\begin{equation}\label{section1_equ5}
		\begin{aligned}
			v^{\alpha - 1}|\nabla v|^{\gamma + 2}\Delta v &= (v^{\alpha - 1}|\nabla v|^{\gamma + 2}v_i)_i - (\alpha - 1)v^{\alpha - 2}|\nabla v|^{\gamma +4} - (\gamma + 2)v^{\alpha - 1}|\nabla v|^{\gamma}v_i v_j v_{ij}\\
			&= (v^{\alpha - 1}|\nabla v|^{\gamma + 2}v_i)_i - (\alpha - 1)v^{\alpha - 2}|\nabla v|^{\gamma +4}\\
			& - (\gamma + 2)v^{\alpha - 1}|\nabla v|^{\gamma + 2}(G_{11} + S\Delta v),\\
			\Rightarrow v^{\alpha - 1}|\nabla v|^{\gamma + 2}\Delta v &= \frac{1}{1 + \gamma S + 2S} (v^{\alpha - 1}|\nabla v|^{\gamma + 2}v_i)_i - \frac{\alpha - 1}{1 + \gamma S + 2S}v^{\alpha - 2}|\nabla v|^{\gamma +4}\\
			&- \frac{\gamma + 2}{1 + \gamma S + 2S} v^{\alpha - 1}|\nabla v|^{\gamma + 2} G_{11}.
		\end{aligned}
	\end{equation}

		\item The term $v^\alpha|\nabla v|^{\gamma}(\Delta v)^2 $: in fact, we have
	\begin{align*}
		&v^{\alpha}|\nabla v|^{\gamma}(\Delta v)^2  = v^{\alpha}|\nabla v|^{\gamma}v_{ii}v_{jj}\\
			=& (v^{\alpha}|\nabla v|^\gamma v_{ii}v_j)_j - \alpha v^{\alpha - 1}|\nabla v|^{\gamma + 2}v_{ii} - v^{\alpha}|\nabla v|^{\gamma}_j v_{ii}v_j - v^\alpha |\nabla v|^\gamma v_{iij}v_j,
	\end{align*}
	and
	\begin{equation*}
		\begin{aligned}
			v^\alpha |\nabla v|^\gamma v_{iij}v_j &= (v^\alpha |\nabla v|^{\gamma }v_{ij}v_j)_i - \alpha v^{\alpha - 1}|\nabla v|^\gamma v_i v_j v_{ij} - v^{\alpha}|\nabla v|^\gamma_iv_{ij}v_j - v^{\alpha}|\nabla v|^{\gamma}v_{ij}^2.
		\end{aligned}
	\end{equation*}	
	As a result, we obtain
	\begin{align*}
		 v^{\alpha}|\nabla v|^{\gamma}(\Delta v)^2&= (v^{\alpha}|\nabla v|^\gamma v_{ii}v_j)_j - \alpha v^{\alpha - 1}|\nabla v|^{\gamma + 2}\Delta v - v^{\alpha}|\nabla v|^{\gamma}_j \Delta v\cdot v_j \\
			&- (v^\alpha |\nabla v|^{\gamma }v_{ij}v_j)_i + \alpha v^{\alpha - 1}v_i v_j v_{ij}|\nabla v|^\gamma + v^{\alpha}|\nabla v|^\gamma_iv_{ij}v_j + v^{\alpha}|\nabla v|^{\gamma}v_{ij}^2\\
			&=  (v^{\alpha}|\nabla v|^{\gamma}\Delta v v_i -  v^{\alpha}|\nabla v|^{\gamma}v_{ij}v_j)_i + v^{\alpha}|\nabla v|^{\gamma}v_{ij}^2 - \alpha v^{\alpha - 1}|\nabla v|^{\gamma + 2}\Delta v\\
			& + \alpha v^{\alpha - 1}|\nabla v|^{\gamma}v_i v_j v_{ij} - v^{\alpha}\Delta v\cdot v_j|\nabla v|^{\gamma}_j + v^{\alpha}v_j v_{ij}|\nabla v|^{\gamma}_i.\\
	\end{align*}

	\begin{align*}
		\Rightarrow & v^{\alpha}|\nabla v|^{\gamma}(\Delta v)^2 = \Big( v^{\alpha}|\nabla v|^{\gamma}\Delta v v_i\Big)_i - ( v^{\alpha}|\nabla v|^{\gamma}v_{ij}v_j)_i + \sum_{i + j > 2} v^{\alpha}|\nabla v|^{\gamma}v_{ij}^2\\
			&+ v^{\alpha}|\nabla v|^{\gamma}G_{11}^2 + 2S v^{\alpha}|\nabla v|^{\gamma}G_{11}\Delta v + S^2 v^{\alpha}|\nabla v|^{\gamma}(\Delta v)^2 - \alpha v^{\alpha - 1}|\nabla v|^{\gamma + 2}\Delta v\\
			& + \alpha v^{\alpha - 1}|\nabla v|^{\gamma + 2}v_{11} - \gamma v^{\alpha}|\nabla v|^{\gamma - 2}v_{i} v_{j} v_{ij}\Delta v + \gamma v^{\alpha}|\nabla v|^{\gamma - 2}v_{j}v_{ij}v_{l}v_{il}\\
			&= \Big( v^{\alpha}|\nabla v|^{\gamma}\Delta v v_i\Big)_i - ( v^{\alpha}|\nabla v|^{\gamma}v_{ij}v_j)_i + \sum_{i + j > 2} v^{\alpha}|\nabla v|^{\gamma}v_{ij}^2\\
			&+ v^{\alpha}|\nabla v|^{\gamma}G_{11}^2 + 2S v^{\alpha}|\nabla v|^{\gamma}G_{11}\Delta v + S^2 v^{\alpha}|\nabla v|^{\gamma}(\Delta v)^2 - \alpha v^{\alpha - 1}|\nabla v|^{\gamma + 2}\Delta v \\
			&+ \alpha v^{\alpha - 1}|\nabla v|^{\gamma + 2}v_{11} - \gamma v^{\alpha}|\nabla v|^{\gamma } v_{11}\Delta v + \gamma v^{\alpha}|\nabla v|^{\gamma }v_{i1}^2\\
			 &= \Big( v^{\alpha}|\nabla v|^{\gamma}\Delta v v_i\Big)_i - ( v^{\alpha}|\nabla v|^{\gamma}v_{ij}v_j)_i + \sum_{i + j > 2} v^{\alpha}|\nabla v|^{\gamma}v_{ij}^2\\
			&+ v^{\alpha}|\nabla v|^{\gamma}G_{11}^2 + 2S v^{\alpha}|\nabla v|^{\gamma}G_{11}\Delta v + S^2 v^{\alpha}|\nabla v|^{\gamma}(\Delta v)^2 - \alpha v^{\alpha - 1}|\nabla v|^{\gamma + 2}\Delta v \\
			&+ \alpha v^{\alpha - 1}|\nabla v|^{\gamma + 2}\Big(G_{11} + S\Delta v\Big) - \gamma v^{\alpha}|\nabla v|^{\gamma }\Big(G_{11} + S\Delta v\Big) \Delta v + \gamma v^{\alpha}|\nabla v|^{\gamma }\Big(G_{11} + S\Delta v\Big)^2\\
			&+ \sum_{i > 1} \gamma v^{\alpha}|\nabla v|^{\gamma}G_{1i}^2.
	\end{align*}

\begin{align*}
		\Rightarrow & \Big(1 - S^2 + \gamma S - \gamma S^2\Big) v^{\alpha}|\nabla v|^{\gamma}(\Delta v)^2\\
	&= \Big( v^{\alpha}|\nabla v|^{\gamma}\Delta v v_i\Big)_i - \Big( v^{\alpha}|\nabla v|^{\gamma}v_{ij}v_j\Big)_i + \sum_{i + j > 2} v^{\alpha}|\nabla v|^{\gamma}v_{ij}^2\\
	&+ \Big(\alpha S - \alpha\Big)v^{\alpha - 1}|\nabla v|^{\gamma + 2}\Delta v + \alpha v^{\alpha - 1}|\nabla v|^{\gamma + 2}G_{11}\\
	&+ \Big(2\gamma S + 2S - \gamma\Big) v^{\alpha} |\nabla v|^{\gamma}G_{11}\Delta v + \Big(1 + \gamma\Big) v^{\gamma}|\nabla v|^{\gamma}G_{11}^2 + \sum_{i > 1} \gamma v^{\alpha}|\nabla v|^{\gamma }G_{i1}^2 \\
	&= \Big( v^{\alpha}|\nabla v|^{\gamma}\Delta v v_i\Big)_i - \Big( v^{\alpha}|\nabla v|^{\gamma}v_{ij}v_j\Big)_i + \sum_{i + j > 2} v^{\alpha}|\nabla v|^{\gamma}v_{ij}^2\\
	&+ \Big(\alpha S - \alpha\Big)\cdot \Bigg[ \frac{1}{1 + \gamma S + 2S} (v^{\alpha - 1}|\nabla v|^{\gamma + 2}v_i)_i - \frac{\alpha - 1}{1 + \gamma S + 2S}v^{\alpha - 2}|\nabla v|^{\gamma +4}\\
	&- \frac{\gamma + 2}{1 + \gamma S + 2S} v^{\alpha - 1}|\nabla v|^{\gamma + 2} G_{11}\Bigg] + \alpha v^{\alpha - 1}|\nabla v|^{\gamma + 2}G_{11}\\
	&+ \Big(2\gamma S + 2S - \gamma\Big) v^{\alpha} |\nabla v|^{\gamma}G_{11}\Delta v + \Big(1 + \gamma\Big) v^{\alpha}|\nabla v|^{\gamma}G_{11}^2 + \sum_{i > 1} \gamma v^{\alpha}|\nabla v|^{\gamma }G_{i1}^2. \\
\end{align*}

	\begin{equation}\label{equ111}
		\begin{aligned}
			\Rightarrow & \Big(1 - S^2 + \gamma S - \gamma S^2\Big)v^{\alpha}|\nabla v|^{\gamma}(\Delta v)^2\\
			&= \Big( v^{\alpha}|\nabla v|^{\gamma}\Delta v v_i\Big)_i - \Big( v^{\alpha}|\nabla v|^{\gamma}v_{ij}v_j\Big)_i + \frac{\alpha S - \alpha}{1 + \gamma S + 2S} \Big(v^{\alpha - 1}|\nabla v|^{\gamma + 2}v_i\Big)_i\\
			&+ \sum_{i + j > 2} v^{\alpha}|\nabla v|^{\gamma}v_{ij}^2  + \frac{\alpha(\alpha - 1)(1 - S)}{1 + \gamma S + 2S}v^{\alpha - 2}|\nabla v|^{\gamma +4}+  \frac{\alpha(\gamma + 3)}{1 + \gamma S + 2S} v^{\alpha - 1}|\nabla v|^{\gamma + 2} G_{11}\\
			&+ \Big(2\gamma S - \gamma + 2S\Big) v^{\alpha} |\nabla v|^{\gamma}G_{11}\Delta v + (1 + \gamma) v^{\alpha}|\nabla v|^{\gamma}G_{11}^2 + \sum_{i > 1} \gamma v^{\alpha}|\nabla v|^{\gamma }G_{i1}^2. \\
		\end{aligned}
	\end{equation}
	Recall that $G_{ij} := v_{ij} - Q\delta_{ij}\Delta v  , i + j > 2$, then
	\begin{align*}
		\sum_{i > 1}v^{\alpha}|\nabla v|^{\gamma}v_{ii}^2 &= \sum_{i > 1}v^{\alpha}|\nabla v|^{\gamma}\Big(G_{ii} + Q\Delta v\Big)^2\\
			&= \sum_{i > 1}v^{\alpha}|\nabla v|^{\gamma}G_{ii}^2 + 2Q\sum_{i > 1}v^{\alpha}|\nabla v|^{\gamma}G_{ii}\Delta v + (n - 1)Q^2 v^{\alpha}|\nabla v|^{\gamma}(\Delta v)^2.
	\end{align*}
	If $(n - 1)Q + S = 1$, then
	\begin{equation*}
		\begin{aligned}
			G_{11} + \sum_{i > 1}G_{ii} = 0,
		\end{aligned}
	\end{equation*}
	and
	\begin{equation*}
		\begin{aligned}
			\sum_{i > 1}v^{\alpha}|\nabla v|^{\gamma}v_{ii}^2 &= \sum_{i > 1}v^{\alpha}|\nabla v|^{\gamma}\Big(G_{ii} + Q\Delta v\Big)^2\\
			&= \sum_{i > 1}v^{\alpha}|\nabla v|^{\gamma}G_{ii}^2 - 2Qv^{\alpha}|\nabla v|^{\gamma}G_{11}\Delta v + (n - 1)Q^2 v^{\alpha}|\nabla v|^{\gamma}(\Delta v)^2.
		\end{aligned}
	\end{equation*}	
	It follows that
	\begin{equation}\label{8.239.37}
		\begin{aligned}
			\Rightarrow & \Big[1 - S^2 + \gamma S - \gamma S^2 - (n - 1)Q^2\Big]v^{\alpha}|\nabla v|^{\gamma}(\Delta v)^2\\
			&= \Big( v^{\alpha}|\nabla v|^{\gamma}\Delta v v_i\Big)_i - ( v^{\alpha}|\nabla v|^{\gamma}v_{ij}v_j)_i + \frac{\alpha S - \alpha}{1 + \gamma S + 2S} (v^{\alpha - 1}|\nabla v|^{\gamma + 2}v_i)_i\\
			&+ \sum_{i + j > 2} v^{\alpha}|\nabla v|^{\gamma}G_{ij}^2  + \frac{\alpha(\alpha - 1)(1 - S)}{1 + \gamma S + 2S}v^{\alpha - 2}|\nabla v|^{\gamma +4}+  \frac{\alpha(\gamma + 3)}{1 + \gamma S + 2S} v^{\alpha - 1}|\nabla v|^{\gamma + 2} G_{11}\\
			&+ \Big(2\gamma S - \gamma + 2S - 2Q\Big) v^{\alpha} |\nabla v|^{\gamma}G_{11}\Delta v + (1 + \gamma) v^{\alpha}|\nabla v|^{\gamma}G_{11}^2 + \sum_{i > 1} \gamma v^{\alpha}|\nabla v|^{\gamma }G_{i1}^2. \\
		\end{aligned}
	\end{equation}

	\begin{remark}
		The divergence term
		\begin{equation*}
			A := (v^{\alpha}|\nabla v|^{\gamma}v_i \Delta v)_i,
		\end{equation*}
		 is important since it is related to the third order derivative of $v$. It will be explained later.
	\end{remark}

		\item The term $A$:

	\begin{equation*}
		\begin{aligned}
			A &= (v^{\alpha}|\nabla v|^{\gamma}v_i \Delta v)_i\\
			&= \alpha v^{\alpha - 1}|\nabla v|^{\gamma + 2}\Delta v + \gamma v^{\alpha}|\nabla v|^{\gamma - 2} v_i v_j v_{ij}\Delta v + v^{\alpha}|\nabla v|^{\gamma}(\Delta v)^2\\
			&+ \underset{A_1}{ v^{\alpha}|\nabla v|^{\gamma}v_i (\Delta v)_i}\\
			&= 	\alpha v^{\alpha - 1}|\nabla v|^{\gamma + 2}\Delta v + \gamma v^{\alpha}|\nabla v|^{\gamma } G_{11}\Delta v + (1 + \gamma S) v^{\alpha}|\nabla v|^{\gamma}(\Delta v)^2\\
			&+ \underset{A_1}{ v^{\alpha}|\nabla v|^{\gamma}v_i (\Delta v)_i}.\\
		\end{aligned}
	\end{equation*}
	Substitute \eqref{1} into $A_1$,
	\begin{equation*}
		\begin{aligned}
			A_1 &=  v^{\alpha}|\nabla v|^{\gamma}v_i \Big( - v^{p}|\nabla v|^q\Big)_i\\
			&= -pv^{\alpha + p - 1}|\nabla v|^{\gamma + q + 2} - q v^{\alpha + p}|\nabla v|^{\gamma + q - 2}v_i v_j v_{ij}.
		\end{aligned}
	\end{equation*}	
	Multiply \eqref{1} by $v^{\alpha - 1}|\nabla v|^{\gamma + 2}$
	\begin{equation}
		\begin{aligned}
			v^{\alpha + p - 1}|\nabla v|^{\gamma + q + 2} = -v^{\alpha - 1}|\nabla v|^{\gamma + 2}\Delta v .
		\end{aligned}
	\end{equation}
	Therefore we have
	\begin{equation}\label{sub2_2}
		\begin{aligned}
			A_1 &= - q v^{\alpha + p}|\nabla v|^{\gamma + q - 2}v_i v_j v_{ij} + p v^{\alpha - 1}|\nabla v|^{\gamma + 2}\Delta v.
		\end{aligned}
	\end{equation}
	On the other hand, we know
	\begin{equation*}
		\begin{aligned}
			&- q v^{\alpha + p}|\nabla v|^{\gamma + q - 2}v_i v_j v_{ij}\\
			&= qv^{\alpha}|\nabla v|^{\gamma - 2}v_i v_j v_{ij} \Delta v\\
			&= q v^{\alpha}|\nabla v|^{\gamma}G_{11}\Delta v +  qS v^{\alpha}|\nabla v|^{\gamma}(\Delta v)^2.\\
		\end{aligned}
	\end{equation*}
	So
	\begin{equation*}
		\begin{aligned}
			A &= (v^{\alpha}|\nabla v|^{\gamma}v_i \Delta v)_i\\
			&= 	\alpha v^{\alpha - 1}|\nabla v|^{\gamma + 2}\Delta v + \gamma v^{\alpha}|\nabla v|^{\gamma } G_{11}\Delta v + (1 + \gamma S) v^{\alpha}|\nabla v|^{\gamma}(\Delta v)^2\\
			&+ \underset{A_1}{ v^{\alpha}|\nabla v|^{\gamma}v_i (\Delta v)_i}\\
			&= \alpha v^{\alpha - 1}|\nabla v|^{\gamma + 2}\Delta v + \gamma v^{\alpha}|\nabla v|^{\gamma } G_{11}\Delta v + (1 + \gamma S)v^{\alpha}|\nabla v|^{\gamma}(\Delta v)^2\\
			&+ p v^{\alpha - 1}|\nabla v|^{\gamma + 2}\Delta v + q v^{\alpha}|\nabla v|^{\gamma}G_{11}\Delta v +  qS v^{\alpha}|\nabla v|^{\gamma}(\Delta v)^2\\
			&= (\gamma + q) v^{\alpha}|\nabla v|^{\gamma } G_{11}\Delta v + (1 + \gamma S + qS) v^{\alpha}|\nabla v|^{\gamma}(\Delta v)^2\\
			&+ ( \alpha + p) v^{\alpha - 1}|\nabla v|^{\gamma + 2}\Delta v.\\
		\end{aligned}
	\end{equation*}
	Substitute \eqref{section1_equ5} into it:
	\begin{equation*}
		\begin{aligned}
			\Rightarrow& - (1 + \gamma S + qS) v^{\alpha}|\nabla v|^{\gamma}(\Delta v)^2 \\
			&= -A + (\gamma + q) v^{\alpha}|\nabla v|^{\gamma } G_{11}\Delta v  + ( \alpha + p) \cdot\Bigg[  \frac{1}{1 + \gamma S + 2S} (v^{\alpha - 1}|\nabla v|^{\gamma + 2}v_i)_i \\
			&- \frac{\alpha - 1}{1 + \gamma S + 2S}v^{\alpha - 2}|\nabla v|^{\gamma +4} - \frac{\gamma + 2}{1 + \gamma S + 2S} v^{\alpha - 1}|\nabla v|^{\gamma + 2} G_{11}\Bigg] ,
		\end{aligned}
	\end{equation*}
	
	and
	\begin{equation}\label{sub3_final4}
		\begin{aligned}
			\Rightarrow& - (1 + \gamma S + qS) v^{\alpha}|\nabla v|^{\gamma}(\Delta v)^2\\
			 &= -A +  \frac{\alpha + p}{1 + \gamma S + 2S} (v^{\alpha - 1}|\nabla v|^{\gamma + 2}v_i)_i  + (\gamma + q) v^{\alpha}|\nabla v|^{\gamma } G_{11}\Delta v \\
			 &- \Big( \alpha + p\Big)\frac{\gamma + 2}{1 + \gamma S + 2S}v^{\alpha - 1}|\nabla v|^{\gamma + 2} G_{11} - \Big(\alpha + p\Big)\frac{\alpha - 1}{1 + \gamma S + 2S}v^{\alpha - 2}|\nabla v|^{\gamma + 4}.
		\end{aligned}
	\end{equation}

	\end{itemize}

	Now we can give the final differential identity.
	Combining \eqref{8.239.37} and \eqref{sub3_final4} to eliminate the term $v^{\alpha}|\nabla v|^{\gamma}(\Delta v)^2$, we obtain
	\begin{equation*}
		\begin{aligned}
			0& = -A +  \frac{\alpha + p}{1 + \gamma S + 2S} (v^{\alpha - 1}|\nabla v|^{\gamma + 2}v_i)_i + (\gamma + q) v^{\alpha}|\nabla v|^{\gamma } G_{11}\Delta v \\
			&- ( \alpha + p)\frac{\gamma + 2}{1 + \gamma S + 2S}v^{\alpha - 1}|\nabla v|^{\gamma + 2} G_{11}  - (\alpha + p )\frac{\alpha - 1}{1 + \gamma S + 2S}v^{\alpha - 2}|\nabla v|^{\gamma + 4}\\
			 &+ \frac{1 + \gamma S + qS}{1 - S^2 + \gamma S - \gamma S^2 - (n - 1)Q^2}\Bigg[ \Big( v^{\alpha}|\nabla v|^{\gamma}\Delta v v_i\Big)_i - ( v^{\alpha}|\nabla v|^{\gamma}v_{ij}v_j)_i \\
			 &+ \frac{\alpha S - \alpha}{1 + \gamma S + 2S} (v^{\alpha - 1}|\nabla v|^{\gamma + 2}v_i)_i\\
			&+ \sum_{i + j > 2} v^{\alpha}|\nabla v|^{\gamma}G_{ij}^2  + \frac{\alpha(\alpha - 1)(1 - S)}{1 + \gamma S + 2S}v^{\alpha - 2}|\nabla v|^{\gamma +4}+  \frac{\alpha(\gamma + 3)}{1 + \gamma S + 2S} v^{\alpha - 1}|\nabla v|^{\gamma + 2} G_{11}\\
			&+ \Big(2\gamma S - \gamma + 2S - 2Q\Big) v^{\alpha} |\nabla v|^{\gamma}G_{11}\Delta v + (1 + \gamma) v^{\alpha}|\nabla v|^{\gamma}G_{11}^2 + \sum_{i > 1} \gamma v^{\alpha}|\nabla v|^{\gamma }G_{i1}^2\Bigg].
		\end{aligned}
	\end{equation*}

	\begin{align}\label{sec1_final}
		\Rightarrow & 0 = W + \Bigg[ \gamma + q + \frac{1 + \gamma S + qS}{1 - S^2 + \gamma S - \gamma S^2 - (n - 1)Q^2}\Big(2\gamma S - \gamma + 2S - 2Q\Big) \Bigg]v^{\alpha}|\nabla v|^{\gamma } G_{11}\Delta v \notag\\
			&+ \frac{1 + \gamma S + qS}{1 - S^2 + \gamma S - \gamma S^2 - (n - 1)Q^2}  v^{\alpha}|\nabla v|^{\gamma}\sum_{i, j = 2}^nG_{ij}  ^2 \notag\\
			& + \frac{1 + \gamma S + qS}{1 - S^2 + \gamma S - \gamma S^2 - (n - 1)Q^2}(1 + \gamma)  v^{\alpha}|\nabla v|^{\gamma}G_{11}^2 \notag\\
			& + \frac{1 + \gamma S + qS}{1 - S^2 + \gamma S - \gamma S^2 - (n - 1)Q^2}(2 + \gamma)\sum_{i > 1}  v^{\alpha}|\nabla v|^{\gamma}G_{1i}^2 \notag\\
			&+ \Bigg[ - ( \alpha + p)\frac{\alpha - 1}{1 + \gamma S + 2S} + \frac{1 + \gamma S + qS}{1 - S^2 + \gamma S - \gamma S^2 - (n - 1)Q^2}\cdot\frac{\alpha(\alpha - 1)(1 - S)}{1 + \gamma S + 2S}   \Bigg]v^{\alpha - 2}|\nabla v|^{\gamma + 4} \notag\\
			&+  \Bigg[ \frac{1 + \gamma S + qS}{1 - S^2 + \gamma S - \gamma S^2 - (n - 1)Q^2}\cdot  \frac{\alpha(\gamma + 3)}{1 + \gamma S + 2S} - (\alpha + p)\frac{\gamma + 2}{1 + \gamma S + 2S}   \Bigg]v^{\alpha - 1}|\nabla v|^{\gamma + 2} G_{11},
	\end{align}
	where $W$ consists of all the divergence terms.
	Then we can rewrite \eqref{sec1_final} as

	\begin{equation}\label{sec2_equ1}
		\begin{aligned}
			0 &= W + a_1 v^{\alpha}|\nabla v|^{\gamma}\sum_{i, j = 2}^nG_{ij}  ^2 + a_4\sum_{i > 1}  v^{\alpha}|\nabla v|^{\gamma}G_{1i}^2  + a_2 v^{\alpha}|\nabla v|^{\gamma}G_{11}^2  + a_3  v^{\alpha - 2}|\nabla v|^{\gamma + 4} \\
			&+ b_1 v^{\alpha - 1}|\nabla v|^{\gamma + 2}G_{11} + b_2 v^{\alpha}|\nabla v|^{\gamma}G_{11}\Delta v ,
		\end{aligned}
	\end{equation}
	with
	\begin{align*}
		a_1 &= \frac{1 + \gamma S + qS}{1 - S^2 + \gamma S - \gamma S^2 - (n - 1)Q^2},\\
			a_2 &=\frac{1 + \gamma S + qS}{1 - S^2 + \gamma S - \gamma S^2 - (n - 1)Q^2}(1 + \gamma) ,\\
			a_3 &= - (\alpha + p)\frac{\alpha - 1}{1 + \gamma S + 2S} + \frac{1 + \gamma S + qS}{1 - S^2 + \gamma S - \gamma S^2 - (n - 1)Q^2}\cdot\frac{\alpha(\alpha - 1)(1 - S)}{1 + \gamma S + 2S} ,\\
			a_4 &= \frac{1 + \gamma S + qS}{1 - S^2 + \gamma S - \gamma S^2 - (n - 1)Q^2}(2 + \gamma),\\
			b_1 &= \frac{1 + \gamma S + qS}{1 - S^2 + \gamma S - \gamma S^2 - (n - 1)Q^2}\cdot  \frac{\alpha(\gamma + 3)}{1 + \gamma S + 2S} - ( \alpha + p)\frac{\gamma + 2}{1 + \gamma S + 2S} ,\\
			b_2 &=  \gamma + q + \frac{1 + \gamma S + qS}{1 - S^2 + \gamma S - \gamma S^2 - (n - 1)Q^2}\Big(2\gamma S - \gamma + 2S - 2Q\Big) .
	\end{align*}

	\section{Conditions}\label{conditions}
	Recalling the results of \cite{MR3959864} , we know the radial solutions of the equation is 
	\begin{equation}
			v_c(r) = c\Big[ Kc^{\frac{(2 - q)^2}{(n - 2)(1 - q)}} + r^{\frac{2 - q}{1 - q}}\Big]^{-\frac{(n - 2)(1 - q)}{2 - q}}.
	\end{equation}
	The result is really inspiring. 
	If we set $v = \tilde K (1 + r^{\beta})^{-(n - 2)\frac{1}{\beta}}$ and $v_1(x_0) = |\nabla v|(x_0)$, then
	\begin{align}
		v_{11} &= \tilde K (n - 2)(1 + r^{\beta})^{-(n - 2)\frac{1}{\beta} - 2} r^{\beta - 2} \Big[ (n - 1) r^{\beta}  - (\beta - 1)\Big], \\
		\Delta v &= - \tilde K (n - 2)(1 + r^{\beta})^{-(n - 2)\frac{1}{\beta} - 2}r^{\beta - 2} \cdot(\beta - 2 + n) , \\
		v^{-1}|\nabla v|^2 &= \tilde K (n - 2)^2(1 + r^{\beta})^{-(n - 2)\frac{1}{\beta} - 2}r^{\beta - 2} \cdot r^{\beta}.
	\end{align}
	So if we hope that 
	\begin{equation}
		G_{11} + cv^{-1}|\nabla v|^2 = 0,
	\end{equation}
	for some constant $c$, then it is necessary to require

	\begin{numcases}{}
		\label{8.23S} S = \frac{\beta - 1}{\beta - 2 + n} = \frac{1}{n - (n - 1)q},\\
		\label{8.23Q}Q = \frac{1 - S}{n - 1} = \frac{1 - q}{n - (n - 1)q}.	
	\end{numcases}
	 At this time, we have

	\begin{equation*}
		\begin{aligned}
			a_1 &= \frac{n - (n - 1)q}{(n - 1)(1 - q)} ,\\
			a_2 &= \frac{n - (n - 1)q}{(n - 1)(1 - q)}(\gamma + 1) ,\\
			a_3 &= -\frac{p(\alpha - 1)}{1 + \gamma S + 2S} ,\\
			b_1 &= \frac{n - (n - 1)q}{(n - 1)(1 - q)}\cdot  \frac{\alpha(\gamma + 3)}{1 + \gamma S + 2S} - ( \alpha + p )\frac{\gamma + 2}{1 + \gamma S + 2S} ,\\
			b_2 &=  \gamma + q + \frac{n - (n - 1)q}{(n - 1)(1 - q)} (2\gamma S - \gamma + 2S - 2Q ) .\\
		\end{aligned}
	\end{equation*}	
	Besides we know that
	\begin{equation*}
		\begin{aligned}
			&(n - 1)\sum_{i > 1}G_{ii}^2 \geq \Big( \sum_{i > 1} G_{ii}\Big)^2 = G_{11}^2,\\
			\Rightarrow &\sum_{i > 1}G_{ii}^2 \geq \frac{1}{n - 1}G_{11}^2.
		\end{aligned}
	\end{equation*}
	Thus \eqref{sec2_equ1} becomes
	\begin{equation}\label{sec2_equ2}
		\begin{aligned}
			0 &\geq W  + \Big(\frac{n}{n - 1} + \gamma\Big) a_1 v^{\alpha}|\nabla v|^{\gamma}G_{11}^2  + a_3  v^{\alpha - 2}|\nabla v|^{\gamma + 4} + b_1 v^{\alpha - 1}|\nabla v|^{\gamma + 2}G_{11}\\
			&+ b_2 v^{\alpha}|\nabla v|^{\gamma}G_{11}\Delta v .\\
		\end{aligned}
	\end{equation}
	Using the identity
	\begin{equation*}
		\begin{aligned}
			-v^{\alpha + p - 1}|\nabla v|^{\gamma + q + 2} &= v^{\alpha - 1}|\nabla v|^{\gamma + 2}\Delta v\\
			&= \frac{1}{1 + \gamma S + 2S} (v^{\alpha - 1}|\nabla v|^{\gamma + 2}v_i)_i - \frac{\alpha - 1}{1 + \gamma S + 2S}v^{\alpha - 2}|\nabla v|^{\gamma +4}\\
			&- \frac{\gamma + 2}{1 + \gamma S + 2S} v^{\alpha - 1}|\nabla v|^{\gamma + 2} G_{11},\\
		\end{aligned}
	\end{equation*}
		then \eqref{sec2_equ2} becomes

	\begin{align}\label{sec2_equ31}
		0 &\geq W  + \Big(\frac{n}{n - 1} + \gamma\Big) a_1 v^{\alpha}|\nabla v|^{\gamma}G_{11}^2  + a_3  v^{\alpha - 2}|\nabla v|^{\gamma + 4} + b_1 v^{\alpha - 1}|\nabla v|^{\gamma + 2}G_{11} \notag\\
			&+ b_2 v^{\alpha}|\nabla v|^{\gamma}G_{11}\Delta v + P\Bigg[ -v^{\alpha + p - 1}|\nabla v|^{\gamma + q + 2} + \frac{\alpha - 1}{1 + \gamma S + 2S}v^{\alpha - 2}|\nabla v|^{\gamma +4} \notag\\
			&+ \frac{\gamma + 2}{1 + \gamma S + 2S} v^{\alpha - 1}|\nabla v|^{\gamma + 2} G_{11} \Bigg] \notag\\
			&= W  + \Big(\frac{n}{n - 1} + \gamma\Big) a_1 v^{\alpha}|\nabla v|^{\gamma}G_{11}^2  + \Bigg(a_3 + P \frac{\alpha - 1}{1 + \gamma S + 2S}\Bigg)  v^{\alpha - 2}|\nabla v|^{\gamma + 4} \notag\\
			&+ \Bigg( b_1 + P \frac{\gamma + 2}{1 + \gamma S + 2S}\Bigg) v^{\alpha - 1}|\nabla v|^{\gamma + 2}G_{11} + b_2 v^{\alpha}|\nabla v|^{\gamma}G_{11}\Delta v - Pv^{\alpha + p - 1}|\nabla v|^{\gamma + q + 2} .
	\end{align}
	We hope that we have already found $S, \gamma, p, q$ such that
	\begin{numcases}{}
		\label{8.23523equ1} b_2 = 0,\\
		P \leq 0,\\
		4 \Big(\frac{n}{n - 1} + \gamma\Big) a_1 \Bigg(a_3 + P \frac{\alpha - 1}{1 + \gamma S + 2S}\Bigg) - \Bigg( b_1 + P \frac{\gamma + 2}{1 + \gamma S + 2S}\Bigg) ^2 > 0.
	\end{numcases}
	In fact, we know the first condition \eqref{8.23523equ1} means that
	\begin{align}
		&b_2 = 0,\\
			 \label{1107b2} \Leftarrow & (\gamma + q)\Big[ 1 - S^2 + \gamma S - \gamma S^2 - (n - 1)Q^2 \Big] + (1 + \gamma S + qS)(2\gamma S - \gamma + 2S - 2Q) = 0 ,\\
			 \Leftarrow& \gamma = (n - 1)q^2 - (n + 1)q.
	\end{align}	
	Besides we can rewrite the identity \eqref{sec2_equ1} as: 
	\begin{equation}
		\begin{aligned}
			0 &= W(\varepsilon) + \varepsilon v^{\alpha}|\nabla v|^{\gamma}v_{i j}^2 + \varepsilon v^{\alpha - 2}|\nabla v|^{\gamma + 4} + \varepsilon v^{\alpha + 2p}|\nabla v|^{\gamma + 2q}\\
			&+ a_1(\varepsilon) v^{\alpha}|\nabla v|^{\gamma}\sum_{i, j = 2}^nG_{ij}  ^2 + a_4(\varepsilon) \sum_{i > 1}  v^{\alpha}|\nabla v|^{\gamma}G_{1i}^2 + a_2(\varepsilon) v^{\alpha}|\nabla v|^{\gamma}G_{11}^2  + a_3(\varepsilon)  v^{\alpha - 2}|\nabla v|^{\gamma + 4} \\
			&+ b_1(\varepsilon) v^{\alpha - 1}|\nabla v|^{\gamma + 2}G_{11} + b_2(\varepsilon) v^{\alpha}|\nabla v|^{\gamma}G_{11}\Delta v, 
		\end{aligned}
	\end{equation}
	Since $\varepsilon$ is small enough, by continuity
 we can choose the same $S, p, P, q$ as before and $\gamma(\varepsilon)$ such that
	\begin{numcases}{}
		b_2 = 0,\\
		P\leq 0,\\
		4 \Big(\frac{n}{n - 1} + \gamma\Big) a_1 \Bigg(a_3 + P \frac{\alpha - 1}{1 + \gamma S + 2S}\Bigg) - \Bigg( b_1 + P \frac{\gamma + 2}{1 + \gamma S + 2S}\Bigg) ^2 > 0.
	\end{numcases}
	then 
	there exists $\varepsilon > 0$, such that
		\begin{equation} \label{sec2_equ3}
		\begin{aligned}
			&\varepsilon v^{\alpha}|\nabla v|^{\gamma}v_{i j}^2 + \varepsilon v^{\alpha - 2}|\nabla v|^{\gamma + 4} + \varepsilon v^{\alpha + 2p}|\nabla v|^{\gamma + 2q}\\
			& \leq B_1(v^{\alpha} |\nabla v|^{\gamma}v_j v_{ij})_i + B_2 (v^{\alpha}|\nabla v|^{\gamma}v_i\Delta v)_i + B_3 (v^{\alpha - 1}|\nabla v|^{\gamma + 2}v_i)_i.
		\end{aligned}
	\end{equation}

	\section{Young inequality}\label{Young}
	This section we will show how to prove $\nabla v \equiv 0$ by \eqref{sec2_equ3}. We must point out that the following integration must be over the area $\Gamma := \{ x \in \mathbb R^n : |\nabla v(x)| > 0\}$  since $\gamma$ may be negative.
	Define $\eta$ is a smooth cut-off function, satisfying that
	\begin{equation*}
		\begin{aligned}
			&\eta \equiv 1 \quad \text{in} \quad B_{\frac{1}{2} R},\\
			&\eta \equiv 0 \quad \text{in} \quad \mathbb R^n\backslash B_{R}.\\
		\end{aligned}
	\end{equation*}
	 Suppose the following holds:
	\begin{equation}\label{int_condtion}
		\begin{aligned}
			&\int_{\partial \Gamma \cap B_{R}} v^{\alpha}|\nabla v|^{\gamma}v_{j}v_{ij}\eta^{\delta}\nu_{i} = 0,\\
			&\int_{\partial \Gamma \cap B_{R}} v^{\alpha + p}|\nabla v|^{\gamma + q}v_{i} \eta^{\delta}\nu_{i} = 0,\\
			&\int_{\partial \Gamma \cap B_{R}} v^{\alpha - 1}|\nabla v|^{\gamma + 2}v_{i} \eta^{\delta}\nu_{i} = 0,\\
		\end{aligned}
	\end{equation}
	where $\nu(x) = (\nu_1, \cdots, \nu_n)$ is the outer normal vector of $\partial \Omega$ at $x$.
	Then multiply \eqref{sec2_equ3} by $\eta^{\delta}$ and integrate over $\Gamma$,

	\begin{equation*}
		\begin{aligned}
			&\varepsilon \int v^{\alpha}|\nabla v|^{\gamma}v_{i j}^2\eta^\delta + \varepsilon \int  v^{\alpha - 2}|\nabla v|^{\gamma + 4}\eta^\delta + \varepsilon \int v^{\alpha + 2p}|\nabla v|^{\gamma + 2q} \eta^{\delta}\\
			&\leq -B_1\delta\int v^{\alpha} |\nabla v|^{\gamma}v_j v_{ij}\eta^{\delta - 1}\eta_i - B_2\delta\int v^{\alpha}|\nabla v|^{\gamma}v_i\Delta v \eta^{\delta - 1}\eta_i -  B_3\delta \int v^{\alpha - 1}|\nabla v|^{\gamma + 2}v_i \eta^{\delta - 1}\eta_i\\
			&= -B_1\delta\int v^{\alpha} |\nabla v|^{\gamma}v_j v_{ij}\eta^{\delta - 1}\eta_i + B_2\delta\int v^{\alpha + p}|\nabla v|^{q + \gamma}v_i \eta^{\delta - 1}\eta_i\\
			&- B_3\delta\int v^{\alpha - 1}|\nabla v|^{\gamma + 2}v_i \eta^{\delta - 1}\eta_i\\
			&\leq \frac{\varepsilon}{2} \int v^{\alpha}|\nabla v|^{\gamma}v_{i j}^2\eta^\delta + C\int v^{\alpha} |\nabla v|^{\gamma + 2}\eta^{\delta - 2}|\nabla \eta|^2\\
			&+ \frac{\varepsilon}{2} \int v^{\alpha + 2p}|\nabla v|^{\gamma + 2q}\eta^{\delta} + C\int v^{\alpha}|\nabla v|^{\gamma + 2}\eta^{\delta - 2}|\nabla \eta|^2\\
			&+ \frac{\varepsilon}{2}\int v^{\alpha - 2}|\nabla v|^{\gamma + 4}\eta^{\delta} + C\int v^{\alpha}|\nabla v|^{\gamma + 2}\eta^{\delta - 2}|\nabla \eta|^2,
		\end{aligned}
	\end{equation*}
	
	\begin{equation}\label{wwz}
		\begin{aligned}
			\Rightarrow & \int v^{\alpha - 2}|\nabla v|^{\gamma + 4}\eta^\delta  + \int v^{\alpha +  2p}|\nabla v|^{\gamma + 2q}\eta^{\delta} \leq C\int v^{\alpha}|\nabla v|^{\gamma + 2}\eta^{\delta - 2}|\nabla \eta|^2.
		\end{aligned}
	\end{equation}	
	Define $p_1, q_1, \sigma_1 > 0$, such that

	\begin{numcases}{}
		\frac{1}{p_1} + \frac{1}{q_1} + \frac{1}{\sigma_1} = 1, \text{ and } p_1, q_1, \sigma_1 > 0,\\
		\label{08235401} \frac{\alpha - 2}{A} = \frac{\gamma + 4}{B} = p_1,\\
		\label{08235402}\frac{\alpha + 2p }{\alpha - A} = \frac{\gamma + 2q}{\gamma + 2 - B} = q_1.
	\end{numcases}
This means that
	\begin{equation*}
		\begin{aligned}
			\Rightarrow& A = \frac{(\alpha - 2)B}{\gamma + 4} = -\frac{(\alpha + 2p)(\gamma + 2 - B)}{\gamma + 2q} + \alpha,\\
			\Rightarrow& B =  (\gamma + 4)\frac{(\gamma + 2)p + (1 - q)\alpha}{(\gamma + 4)p + (2 - q)\alpha + \gamma + 2q},\\
			&A =  (\alpha - 2)\frac{(\gamma + 2)p + (1 - q)\alpha}{(\gamma + 4)p + (2 - q)\alpha + \gamma + 2q}.\\
		\end{aligned}
	\end{equation*}

	We need to check that
	
	\begin{numcases}{}
		 \label{sec2_cond1} 1 - \frac{2}{n} < \frac{1}{p_1} + \frac{1}{q_1} < 1,\\
		\label{sec2_cond2} p_1, q_1 > 0.
	\end{numcases}
	If they are satisfied, then
	\begin{equation*}
		\begin{aligned}
			 & \int v^{\alpha - 2}|\nabla v|^{\gamma + 4}\eta^\delta + \int v^{\alpha +  2p}|\nabla v|^{\gamma + 2q}\eta^{\delta}\\
			 & \leq C\int v^{\alpha}|\nabla v|^{\gamma + 2}\eta^{\delta - 2}|\nabla \eta|^2\\
			 &\leq \frac{1}{2} \int v^{\alpha - 2}|\nabla v|^{\gamma + 4}\eta^\delta + \frac{1}{2}\int v^{\alpha +  2p}|\nabla v|^{\gamma + 2q}\eta^{\delta} + C\int \eta^{\delta - 2\sigma_1}|\nabla \eta|^{2\sigma_1},\\
			\Rightarrow &  \int v^{\alpha - 2}|\nabla v|^{\gamma + 4}\eta^\delta + \int v^{\alpha +  2p}|\nabla v|^{\gamma + 2q}\eta^{\delta} \leq CR^{n - 2\sigma_1} \rightarrow 0, \text{as R tends to infinity}.
		\end{aligned}
	\end{equation*}
	Thus $v \equiv const$. Next we will show that there exist $\alpha, \gamma, \delta, p, q$ such that the above conditions will be satisfied at the same time. 
	
	\section{Proof of theorems}
	There are three cases for different $0 < q < 2$:
	\subsection{ $0 < q \leq \frac{1}{n - 1}$}
		We prefer to focus more on the case $q > 0$. If $0 < q \leq \frac{1}{n - 1}$ and $\varepsilon_1 > 0$ is small, by \eqref{8.23S} we choose
		\begin{equation}
			\gamma = (n - 1)q^2 - (n + 1)q + \varepsilon_1 \geq -\frac{n}{n - 1} + \varepsilon_1.
		\end{equation}
		If $\varepsilon_1 = 0$, then solving the quadratic equation \eqref{1107b2} with respect to $S$,  we get $S = \frac{1}{n - (n - 1)q}$ or $ \frac{2 - q}{n q - (n - 1)q^2} $. So for $\varepsilon_1 > 0$ small enough, there always exists one solution
		\begin{equation}
			S = \frac{1}{n - (n - 1)q} + O(\varepsilon_1),
		\end{equation}	
		such that $b_2 = 0$.
		Here $O(\varepsilon_1)$ means the term will tend to zero as $\varepsilon_1$ tends to zero.
		Since $\gamma + q + 1 > 0$,
		\begin{equation*}
			\begin{aligned}
				&\Bigg| \int_{\partial \Gamma \cap B_{R}} v^{\alpha + p}|\nabla v|^{\gamma + q}v_{i} \eta^{\delta}\nu_{i}\Bigg|\\
				&\leq \int_{\partial \Gamma \cap B_{R}} v^{\alpha + p}|\nabla v|^{\gamma + q + 1} \eta^{\delta} = 0.
			\end{aligned}
		\end{equation*}
		Similarly, we can show that conditions \eqref{int_condtion} are satisfied. Then we need that
		\begin{numcases}{}
			P \leq 0,\\
			\label{8.23657} 4 \Big(\frac{n}{n - 1} + \gamma\Big) a_1 \Bigg(a_3 + P \frac{\alpha - 1}{1 + \gamma S + 2S}\Bigg) - \Bigg( b_1 + P \frac{\gamma + 2}{1 + \gamma S + 2S}\Bigg) ^2 > 0.
		\end{numcases}
		In fact, we know
		\begin{align}
			a_1 &= \frac{1 + \gamma S + qS}{1 - S^2 + \gamma S - \gamma S^2 - (n - 1)Q^2},\\
				a_3 &= -\frac{p(\alpha - 1)}{1 + \gamma S + 2S} + \frac{a_1(1 - S) - 1}{1 + \gamma S + 2S}\alpha(\alpha - 1) = -\frac{p(\alpha - 1)}{1 + \gamma S + 2S} + O(\varepsilon_1) ,\\
			b_1 &= a_1\cdot  \frac{\alpha(\gamma + 3)}{1 + \gamma S + 2S} - ( \alpha + p )\frac{\gamma + 2}{1 + \gamma S + 2S} .
		\end{align}
		Define
		\begin{align}
			s(p) &:= \frac{(2 - q)a_3}{a_1(\alpha - 1)} + 1 \notag\\
			&= \frac{(2 - q)}{a_1}\left[ \frac{-p}{1 + \gamma S + 2S} + \frac{a_1(1 - S) - 1}{1 + \gamma S + 2S} \right] + 1\\
			& = -\frac{(n - 1)(2 - q)(1 - q) p}{2 + n   - 2n q + (n - 1)q^2} + 1 + O(\varepsilon_1),
		\end{align}
		then
		\begin{align}
			a_3(p) &= \frac{a_1}{2 - q}(s - 1)(\alpha - 1) ,\\
			b_1(p) &= \frac{a_1}{2 - q}\Big[ (s - 1)(\gamma + 2) + (2 - q)\alpha\Big].
		\end{align}
		So it follows that
		\begin{align*}
			&4 \Big(\frac{n}{n - 1} + \gamma\Big) a_1 \Bigg(a_3 + P \frac{\alpha - 1}{1 + \gamma S + 2S}\Bigg) - \Bigg( b_1 + P \frac{\gamma + 2}{1 + \gamma S + 2S}\Bigg) ^2\\
				&= 4 \Big(\frac{n}{n - 1} + \gamma\Big) a_1 a_3(p - P) - b_1(p - P)^2\\
				&= 4 \Big(\frac{n}{n - 1} + \gamma\Big)  \frac{a_1^2}{2 - q}(s - 1)(\alpha - 1)  - \frac{a_1^2}{(2 - q)^2}\Big[ (s - 1)(\gamma + 2) + (2 - q)\alpha\Big]^2\\
				=&a_1^2 \Bigg[ -\alpha^2 + \frac{2}{2 - q}(s - 1)\left( \frac{2}{n - 1} + \gamma     \right)\alpha - \frac{4}{2 - q} \Big(\frac{n}{n - 1} + \gamma\Big)   (s - 1) - \frac{1}{(2 - q)^2} (s - 1)^2(\gamma + 2)^2 \Bigg].
		\end{align*}

		Let 
		\begin{align}\label{alpha}
			\alpha &= \frac{s - 1}{2 - q}\left( \frac{2}{n - 1} + \gamma     \right) \notag\\
			&= \frac{1}{a_1} \Bigg[ -\frac{p - P}{1 + \gamma S + 2S} + \frac{a_1(1 - S) - 1}{1 + \gamma S + 2S}\alpha \Bigg]\left( \frac{2}{n - 1} + \gamma     \right),
		\end{align}
		then we get
		\begin{align}
			&a_1^{-2}\Bigg[ 4 \Big(\frac{n}{n - 1} + \gamma\Big) a_1 \Bigg(a_3 + P \frac{\alpha - 1}{1 + \gamma S + 2S}\Bigg) - \Bigg( b_1 + P \frac{\gamma + 2}{1 + \gamma S + 2S}\Bigg) ^2 \Bigg] \notag\\
			=&\frac{(s - 1)^2}{(2 - q)^2}\left( \frac{2}{n - 1} + \gamma     \right)^2 - \frac{4}{2 - q} \Big(\frac{n}{n - 1} + \gamma\Big)   (s - 1) - \frac{1}{(2 - q)^2} (s - 1)^2(\gamma + 2)^2\notag\\
			=&\frac{s - 1}{2 - q} \left[ \frac{s - 1}{2 - q}\left( \frac{2}{n - 1} + \gamma     \right)^2 - 4\Big(\frac{n}{n - 1} + \gamma\Big)    - \frac{s - 1}{2 - q} (\gamma + 2)^2 \right] \notag\\
			=&4\frac{s - 1}{2 - q} \left[ \frac{s - 1}{2 - q}\left( \frac{2 - n}{n - 1}\gamma + \frac{2n - n^2}{(n - 1)^2}      \right) - \Big(\frac{n}{n - 1} + \gamma\Big)  \right] \notag\\
			=&4\frac{s - 1}{2 - q}\Big(\frac{n}{n - 1} + \gamma\Big) \left[ \frac{2 - n}{n - 1} \cdot \frac{s - 1}{2 - q} - 1  \right] \notag\\
			=&4 \left[ -\frac{(n - 1)(1 - q) (p - P)}{2 + n   - 2n q + (n - 1)q^2}  + O(\varepsilon_1) \right] \Big(\frac{n}{n - 1} + \gamma\Big) \left[ \frac{(1 - q)(2 - n) (p - P)}{2 + n   - 2n q + (n - 1)q^2} + 1 + O(\varepsilon_1)  \right] .
		\end{align}
		Thus we deduce that if
	\begin{numcases}{}
		\alpha =  \frac{s - 1}{2 - q}\left( \frac{2}{n - 1} + \gamma     \right) =: \alpha_1 ,\\
		0 < p - P < \frac{1}{n - 2}\Bigg[ n + \frac{2 - q}{1 - q} - q(n - 1) \Bigg] =:p_*,
	\end{numcases}
		there exists $\varepsilon > 0$ much smaller than $\varepsilon_1$, such that
		\begin{equation}
		\begin{aligned}
			&\varepsilon v^{\alpha}|\nabla v|^{\gamma}v_{ij}^2 + \varepsilon v^{\alpha - 2}|\nabla v|^{\gamma + 4} + \varepsilon v^{\alpha + 2p}|\nabla v|^{\gamma +  2q}\\
			& \leq B_1(v^{\alpha} |\nabla v|^{\gamma}v_j v_{ij})_i + B_2 (v^{\alpha}|\nabla v|^{\gamma}v_i\Delta v)_i + B_3 (v^{\alpha - 1}|\nabla v|^{\gamma + 2}v_i)_i.
		\end{aligned}
	\end{equation}
	If we choose $p - P = p_* - \varepsilon$ and $\alpha = \alpha_1$ for $0 < \varepsilon << \varepsilon_1$ very small , then 
	\begin{equation}
		\begin{aligned}
			\alpha_1 &= -\frac{n - 1}{n - 2}\Big(\gamma + \frac{2}{n - 1}\Big) + O(\varepsilon_1).
		\end{aligned}
	\end{equation}
	For simplicity, we ignore the term $O(\varepsilon_1)$ if there is no confusion.
		For fixed $q, \gamma, \alpha$, by \eqref{08235401} and \eqref{08235402} we have
	\begin{equation}
		\begin{aligned}
			B(p) &= (\gamma + 4)\frac{(\gamma + 2)p + (1 - q)\alpha_1}{(\gamma + 4)p + (2 - q)\alpha_1 + \gamma + 2q}\\
			&=: (\gamma + 4)\frac{ap + b}{cp + d}\\
			&= (\gamma + 4)\left(\frac{a}{c} + \frac{bc - ad}{c^2p + cd}\right),
		\end{aligned}
	\end{equation}
	where
	\begin{align}
		a :&= \gamma + 2 ,\\
		b :&= ( 1 - q)\alpha_1,\\
		c :&= \gamma + 4,\\
		d :&= (2 - q)\alpha_1 + \gamma + 2q.
	\end{align}
	Then it holds that
	\begin{equation*}
		\begin{aligned}
			bc - ad &= (\gamma + 4) (1 - q)\alpha - (\gamma + 2)\Big[ (2 - q)\alpha_1 + \gamma + 2q \Big]\\
			&= -(\gamma + 2q)(\alpha_1 + \gamma + 2)\\
			&= -(\gamma + 2q)\left( - \frac{1}{n - 2}\gamma + \frac{2n - 6}{n - 2} \right).\\
		\end{aligned}
	\end{equation*}
	Thus when $0 < q \leq \frac{1}{n - 1}$, 
	\begin{equation*}
		bc - ad > 0.
	\end{equation*}
	We deduce that $B(p)$ is decreasing when
	\begin{equation*}
		cp + d = (\gamma + 4)p + (2 - q)\alpha_1 + \gamma + 2q > 0.
	\end{equation*}
	\begin{lemma}
		When $1 - q < p < p_*$, $B(p)$ satisfies the conditions   \eqref{sec2_cond1} and
		\eqref{sec2_cond2}.
	\end{lemma}

\begin{proof}
	Firstly, we know that when $1 - q < p < p_*$
	\begin{equation*}
		\begin{aligned}
			&(\gamma + 4)p + (2 - q)\alpha_1 + \gamma + 2q \\
		&> (\gamma + 4)(1 - q) + (2 - q)\alpha_1 + \gamma + 2q \\
		&= (\gamma + 4)(1 - q) - (2 - q)\frac{n - 1}{n - 2}\left(\gamma + \frac{2}{n - 1}\right) + \gamma + 2q \\
		&= \frac{\gamma}{n - 2}(q - 2) + \frac{n - 3}{n - 2}(4 - 2q) > 0.
		\end{aligned}
	\end{equation*}
	So $B(p)$ is decreasing. Besides, we know
\begin{equation*}
	\begin{aligned}
		\frac{1}{p_1} + \frac{1}{q_1} &= \frac{2(q - 2)}{(\gamma + 4)(\gamma + 2q)}B + \frac{\gamma + 2}{\gamma + 2q},
	\end{aligned}
\end{equation*}
	and 
	\begin{equation*}
		\frac{2(q - 2)}{(\gamma + 4)(\gamma + 2q)} > 0.
	\end{equation*}
	By direct computations, we have
	\begin{align}
			&\frac{1}{p_1} + \frac{1}{q_1} > 1 - \frac{2}{n},\\
		\Leftrightarrow& \frac{2(q - 2)}{(\gamma + 4)(\gamma + 2q)}B + \frac{\gamma + 2}{\gamma + 2q} > 1 - \frac{2}{n},\\
		\Leftrightarrow& B  > \frac{(\gamma + 4)} {2(q - 2)} \left[ \frac{n - 2}{n}(\gamma + 2q) - \gamma - 2 \right] .
	\end{align}
		In fact, it holds that
	\begin{equation}
		\begin{aligned}
			RHS &= \frac{(\gamma + 4)} {2(q - 2)} \left[ \frac{n - 2}{n}(n - 1)q(q - 1) - (n - 1)q^2 + (n + 1)q - 2 \right] \\
			&= \frac{(\gamma + 4)} {2(q - 2)} \left[ -\frac{ 2}{n}(n - 1)q^2 + \left(4 - \frac{2}{n}\right)q - 2 \right] \\
			&= \frac{(\gamma + 4)} {n(2 - q)}\Big[ n - (n - 1)q\Big](1 - q).
		\end{aligned}
	\end{equation}
		Recalling that 
		\begin{align}
			p_*& = \frac{(n - 1)q^2 - 2nq + n + 2}{(n - 2)(1 - q)},\\
			\alpha_1 &= -\frac{n - 1}{n - 2}\Big(\gamma + \frac{2}{n - 1}\Big),
		\end{align}
		we know
	\begin{equation}
		\begin{aligned}
			B(p_*) &= (\gamma + 4)\frac{(\gamma + 2)p_* + (1 - q)\alpha_1}{(\gamma + 4)p_* + (2 - q)\alpha_1 + \gamma + 2q}\\
			 &= (\gamma + 4)\frac{(\gamma + 2) \frac{(n - 1)q^2 - 2nq + n + 2}{(n - 2)(1 - q)}  -\frac{n - 1}{n - 2}\Big(\gamma + \frac{2}{n - 1}\Big) (1 - q)}{(\gamma + 4) \frac{(n - 1)q^2 - 2nq + n + 2}{(n - 2)(1 - q)}  -\frac{n - 1}{n - 2}\Big(\gamma + \frac{2}{n - 1}\Big) (2 - q) + \gamma + 2q}\\
			 &= (\gamma + 4)\frac{ (-2q + 3)\gamma + 2(n - 2)q^2 - (4n - 4)q + 2n + 2 }{ (2 - q)\gamma + 2(n - 1)q^2 - (6n - 2)q + 4n + 4}\\
			  &= (\gamma + 4)\frac{ (-2n + 2)q^3 + (7n - 5)q^2 + (-7n + 1)q + 2n + 2 }{ (2 - q)\Big[(n - 1)q^2 - (n + 1)q\Big] + 2(q - 2)\Big[(n - 1)q - n - 1\Big]}\\
			   &= \frac{\gamma + 4}{2 - q}\frac{(1 - q)\Big[n + 1 - (2n - 2)q \Big] }{ \Big[n + 1 - (n - 1)q \Big]}.\\
		\end{aligned}
	\end{equation}
So we get when $0 \leq q \leq \frac{1}{n - 1}$ and $1 - q < p < p_*$
\begin{align}
	&B(p) > B(p_*) \geq RHS,\\
	\Rightarrow &\frac{1}{p_1} + \frac{1}{q_1} > 1 - \frac{2}{n}.
\end{align}
Besides, we also have
\begin{align}
	&\frac{1}{p_1} + \frac{1}{q_1} < 1,\\
	\Leftrightarrow & p > 1 - q.
\end{align}
	Finally we need to check that 
	\begin{equation*}
		\begin{aligned}
			B(p) > \gamma + 2.
		\end{aligned}
	\end{equation*}
	In fact, we have
	\begin{equation*}
		\begin{aligned}
			& B(p) > \gamma + 2,\\
			\Leftrightarrow&  (\gamma + 4)\Big[(\gamma + 2)p + (1 - q)\alpha\Big] > \Big[(\gamma + 4)p + (2 - q)\alpha + \gamma + 2q\Big](\gamma + 2),\\
			\Leftrightarrow & (\gamma + 2q)(\gamma + 2 + \alpha) < 0.
		\end{aligned}
	\end{equation*}
	This is proved before. Therefore, we obtain that $p_1, q_1 > 0$.
	
\end{proof}

		\subsection{ $q > \frac{1}{n - 1}$}
		 If $q > \frac{1}{n - 1}$, this case is complex since we can not choose $\gamma = (n - 1)q^2 - (n + 1)q$ any more. In this case we choose $S = 0$. At this time, we have

	 \begin{equation*}
		\begin{aligned}
			a_1 &= \frac{n - 1}{n - 2},\\
			a_3 &= - p(\alpha - 1) + \frac{1}{n - 2}\cdot \alpha(\alpha - 1) ,\\
			b_1 &= \frac{n - 1}{n - 2}\cdot\alpha(\gamma + 3)- ( \alpha + p)(\gamma + 2) ,\\
			b_2 &=  \gamma + q + \frac{n - 1}{n - 2}\Big( - \gamma - \frac{2}{n - 1}\Big).\\
		\end{aligned}
	\end{equation*}

	If $b_2 = 0$, then
	\begin{equation*}
		\begin{aligned}
			b_2 &= -\frac{1}{n - 2}\gamma + q - \frac{2}{n - 2} = 0,\\
			\Rightarrow \gamma &= (n - 2)q - 2 > -\frac{n}{n - 1}.
		\end{aligned}
	\end{equation*}
	
	Choose $\gamma = (n - 2)q - 2$ and define $B_0 := a_1\left(\frac{n}{n - 1} + \gamma \right)$ and

	 \begin{equation*}
		\begin{aligned}
			a_1 &= \frac{n - 1}{n - 2},\\
			a_3 &= - p(\alpha - 1) + \frac{1}{n - 2}\cdot \alpha(\alpha - 1) ,\\
			b_1 &= q\alpha + \frac{n - 1}{n - 2}\alpha - (n - 2)pq ,\\
			b_2 &= 0 ,\\
			B_0 &= (n - 1)q - 1.
		\end{aligned}
	\end{equation*}
	By direct computation, conditions \eqref{int_condtion} are satisfied.
	Then using $S = 0$
	
	\begin{align*}
		&4 B_0 \Bigg(a_3 + P \frac{\alpha - 1}{1 + \gamma S + 2S}\Bigg) - \Bigg( b_1 + P \frac{\gamma + 2}{1 + \gamma S + 2S}\Bigg) ^2\\
			&= 4\Big[ (n - 1)q - 1 \Big]\cdot\Bigg[ \frac{1}{n - 2}\alpha(\alpha - 1) - (p - P)(\alpha - 1) \Bigg] - \Bigg[q\alpha + \frac{n - 1}{n - 2}\alpha - (n - 2)(p - P)q\Bigg]^2\\
			&= 4\Big[ (n - 1)q - 1 \Big]\cdot\Bigg[ \frac{1}{n - 2}\alpha^2 - \Big(p - P + \frac{1}{n - 2}\Big)\alpha + p - P \Bigg] - \Bigg[ \Big(q + \frac{n - 1}{n - 2}\Big)\alpha - (n - 2)(p - P)q\Bigg]^2\\
			&= \Bigg\{4\Big[ (n - 1)q - 1 \Big]\cdot \frac{1}{n - 2} - \Big(q + \frac{n - 1}{n - 2}\Big)^2 \Bigg\}\alpha^2\\
			&+ \Bigg\{ -4\Big[ (n - 1)q - 1 \Big]\cdot \Big(p - P + \frac{1}{n - 2}\Big) + 2\Big(q + \frac{n - 1}{n - 2}\Big)(n - 2)(p - P)q\Bigg\}\alpha\\
			&+ 4(p - P)\Big[ (n - 1)q - 1 \Big] - (n - 2)^2(p - P)^2q^2\\
			&= \Bigg[ -q^2 + 2q\frac{n - 1}{n - 2} - \frac{4}{n - 2} - \Big(\frac{n - 1}{n - 2}\Big)^2 \Bigg]\alpha^2\\
			&+ \Bigg[ -2(n - 1)(p - P)q - 4\frac{n - 1}{n - 2}q + 4(p - P) + \frac{4}{n - 2} + 2(n - 2)(p - P)q^2 \Bigg]\alpha\\
			&+ 4(p - P)\Big[ (n - 1)q - 1 \Big] - (n - 2)^2(p - P)^2q^2.
	\end{align*}
		Since 
	\begin{equation*}
		\begin{aligned}
			& -q^2 + 2q\frac{n - 1}{n - 2} - \frac{4}{n - 2} - \Big(\frac{n - 1}{n - 2}\Big)^2\\
			&= -\left(q - \frac{n - 1}{n - 2}\right)^2 - \frac{4}{n - 2} < 0,
		\end{aligned}
	\end{equation*}
	if we choose 
	\begin{equation*}
		\begin{aligned}
			\alpha &= \Bigg[ q^2 - 2q\frac{n - 1}{n - 2} + \frac{4}{n - 2} + \Big(\frac{n - 1}{n - 2}\Big)^2 \Bigg]^{-1} \Bigg[ -(n - 1)(p - P)q - 2\frac{n - 1}{n - 2}q + 2(p - P)\\
			& + \frac{2}{n - 2} + (n - 2)(p - P)q^2 \Bigg]\\
			&= :\alpha_{2},
		\end{aligned}
	\end{equation*}
	then we only need that
	\begin{align*}
			0&<\Bigg\{ \Big[ 2(n - 2)q^2 - 2(n - 1)q + 4\Big](p - P) - 4\frac{n - 1}{n - 2}q + \frac{4}{n - 2} \Bigg\}^2\\
		&- 4 \Bigg[ -q^2 + 2q\frac{n - 1}{n - 2} - \frac{4}{n - 2} - \Big(\frac{n - 1}{n - 2}\Big)^2 \Bigg]\cdot\Bigg\{ 4(p - P)\Big[ (n - 1)q - 1 \Big] - (n - 2)^2(p - P)^2q^2 \Bigg\}\\
		&=\Bigg\{ \Big[ 2(n - 2)q^2 - 2(n - 1)q + 4\Big]^2 + 4 (n - 2)^2q^2\Bigg[ -q^2 + 2q\frac{n - 1}{n - 2} - \frac{4}{n - 2} - \Big(\frac{n - 1}{n - 2}\Big)^2 \Bigg] \Bigg\}(p - P)^2\\
		&+\Bigg\{ 16 \Big[ (n - 2)q^2 - (n - 1)q + 2\Big]\cdot\Big(- \frac{n - 1}{n - 2}q + \frac{1}{n - 2}\Big)\\
		&- 16\Big[ (n - 1)q - 1 \Big] \Bigg[ -q^2 + 2q\frac{n - 1}{n - 2} - \frac{4}{n - 2} - \Big(\frac{n - 1}{n - 2}\Big)^2 \Bigg] \Bigg\}(p - P)\\
		&+ \Big(- 4\frac{n - 1}{n - 2}q + \frac{4}{n - 2}\Big)^2,
	\end{align*}
	\begin{equation*}
		\begin{aligned}
			\Leftrightarrow 0 &< 16\Big[1 - (n - 1)q\Big](p - P)^2 + 16 \Big[1 - (n - 1)q\Big]\cdot\Big[\frac{n - 1}{n - 2}q - \frac{n^2 - 3}{(n - 2)^2}\Big](p - P)\\
			&+ 16\Big(- \frac{n - 1}{n - 2}q + \frac{1}{n - 2}\Big)^2,
		\end{aligned}
	\end{equation*}
	
	\begin{equation}
		\begin{aligned}
			\Leftrightarrow 0 &> (p - P)^2 + \Big[\frac{n - 1}{n - 2}q - \frac{n^2 - 3}{(n - 2)^2}\Big](p - P) + \frac{1 - (n - 1)q}{(n - 2)^2}.
		\end{aligned}
	\end{equation}
	For fixed $q$, let $p_2 > 0$ be one solution to 
	\begin{equation}\label{8.231125}
		\mathbb H(p, q) := p^2 + \Big[\frac{n - 1}{n - 2}q - \frac{n^2 - 3}{(n - 2)^2}\Big]p + \frac{1 - (n - 1)q}{(n - 2)^2} = 0.
	\end{equation}
	Choose $p - P = p_2$ and $\alpha = \alpha_2$. Recall that we have already showed that for fixed $q, \gamma, \alpha$
	\begin{equation}
		\begin{aligned}
			B(p) &= (\gamma + 4)\frac{(\gamma + 2)p + (1 - q)\alpha_2}{(\gamma + 4)p + (2 - q)\alpha_2 + \gamma + 2q}\\
			&=: (\gamma + 4) \frac{ap + b}{cp + d}\\
			&= (\gamma + 4)\left(\frac{a}{c} + \frac{bc - ad}{c^2p + cd}\right),
		\end{aligned}
	\end{equation}
	
	\begin{equation*}
		\begin{aligned}
			bc - ad &= (\gamma + 4) (1 - q)\alpha_2 - (\gamma + 2)\Big[ (2 - q)\alpha_2 + \gamma + 2q \Big]\\
			&= -(\gamma + 2q)(\alpha_2 + \gamma + 2).\\
		\end{aligned}
	\end{equation*}
	
	\begin{claim}\label{09021818equ1}
		For fixed $n > 2$, $p_2(q)$ is decreasing.
	\end{claim}
	\begin{proof}
		By the definition of $p_2(q)$, we have
		\begin{equation}
			\begin{aligned}
				0 &= p_2^2 + \Big[\frac{n - 1}{n - 2}q - \frac{n^2 - 3}{(n - 2)^2}\Big]p_2 + \frac{1 - (n - 1)q}{(n - 2)^2}\\
				&= \frac{n - 1}{n - 2}\left(p_2 - \frac{1}{n - 2}\right)q + p_2^2 - \frac{n^2 - 3}{(n - 2)^2}p_2 + \frac{1}{(n - 2)^2},\\
			\end{aligned}
		\end{equation}
		\begin{align}
		\Rightarrow	0 &= \frac{n - 1}{n - 2}q dp_2 + \frac{n - 1}{n - 2}\left(p_2 - \frac{1}{n - 2}\right) dq + 2p_2 dp_2 - \frac{n^2 - 3}{(n - 2)^2} dp_2.
		\end{align}
		Then we get
		\begin{equation}
			\begin{aligned}
				\frac{d p_2}{dq} = -\left[ \frac{n - 1}{n - 2}q   + 2p_2  - \frac{n^2 - 3}{(n - 2)^2}\right]^{-1}\frac{n - 1}{n - 2}\left(p_2 - \frac{1}{n - 2}\right).
			\end{aligned}
		\end{equation}
		Besides by direct computation and \eqref{8.231125} , we obtain $\mathbb H( \frac{1}{n - 2}, q) < 0$ for any $0 \leq q < 2$. So $p_2 > \frac{1}{n - 2}$ and
		\begin{equation}
			\begin{aligned}
				&\left[ \frac{n - 1}{n - 2}q   + 2p_2  - \frac{n^2 - 3}{(n - 2)^2}\right]p_2\\
				=& \left( \frac{n - 1}{n - 2}q   + 2p_2 \right)p_2 -  \frac{n - 1}{n - 2}\left(p_2 - \frac{1}{n - 2}\right)q - p_2^2  - \frac{1}{(n - 2)^2}\\
				=& p_2^2 + \frac{n - 1}{(n - 2)^2}q  - \frac{1}{(n - 2)^2} > \frac{n - 1}{(n - 2)^2}q \geq 0.
			\end{aligned}
		\end{equation} Hence $p_2(q)$ is decreasing with respect to $q$.
	\end{proof}

	\begin{claim}
		We have $\alpha_2 + \gamma + q > 0$. Hence $\alpha_2 + \gamma + 2 > 0$.
	\end{claim}
	\begin{proof}
		\begin{equation}
			\begin{aligned}
				&\alpha_2 + \gamma + q > 0, \\
				\Leftrightarrow & \Big[-(n - 1)p_2 q - 2q\frac{n - 1}{n - 2} + 2p_2 + \frac{2}{n - 2} + (n - 2)p_2 q^2\Big] \\
				&+ \Big[ (n - 1)q - 2\Big] \Bigg[ q^2 - 2q\frac{n - 1}{n - 2} + \frac{4}{n - 2} + \Big(\frac{n - 1}{n - 2}\Big)^2 \Bigg] > 0,\\
			\end{aligned}
		\end{equation}
		\begin{equation}\label{8.23.10.42}
			\begin{aligned}
				\Leftrightarrow &\mathbb F :=  (n - 1)q^3 + \Big[ (n - 2)p_2  - 2 \frac{n^2 - n - 1}{n - 2} \Big]q^2 + \Big[ -(n - 1)p_2 + (n - 1) \frac{n^2 + 4n - 11}{(n - 2)^2} \Big]q\\
				& + 2p_2 + \frac{-2n^2 - 2n + 10}{(n - 2)^2} > 0.
			\end{aligned}
		\end{equation}
	By the proof of Claim \ref{09021818equ1}, we get that $\frac{1}{n - 2} < p_2(q) < p_2(\frac{1}{n - 1}) = \frac{n^2 - n - 1}{(n - 2)^2}$ when $\frac{1}{n - 1} < q < 2$.
	Define 
	\begin{equation}\label{09021826equ1}
		w:= \frac{n - 2}{n - 1}\left(p_2 - \frac{1}{n - 2}\right), \Rightarrow p_2 = \frac{(n - 1)w + 1}{n - 2},
	\end{equation}
	then 
	\begin{equation}
		0 < w < \frac{n - 1}{n - 2}.
	\end{equation}
	Use the fact that $\mathbb H(p_2, q) = 0$ and \eqref{8.231125} :
	\begin{equation}\label{09021826equ2}
		(n - 2)q = - (n - 2)w +  w^{-1} + n - 1.
	\end{equation}
	Substitute \eqref{09021826equ1} and \eqref{09021826equ2} into \eqref{8.23.10.42}: 
	
	\begin{align*}
		\mathbb F &=  (n - 1)q^3 + \Big[ (n - 1)w + 1  - 2 \frac{n^2 - n - 1}{n - 2} \Big]q^2 + \Big[ -(n - 1) \frac{(n - 1)w + 1}{n - 2} + (n - 1) \frac{n^2 + 4n - 11}{(n - 2)^2} \Big]q\\
		& + 2 \frac{(n - 1)w + 1}{n - 2} + \frac{-2n^2 - 2n + 10}{(n - 2)^2} \\
		&=  (n - 1)q^3 + \Big[ (n - 1)w  -  \frac{2n^2 - 3n}{n - 2} \Big]q^2 +  \frac{n - 1}{(n - 2)^2}\Big[ -(n - 1)(n - 2)w + n^2 + 3n - 9 \Big]q\\
		& +  \frac{2(n - 1)w }{n - 2} + \frac{-2n^2 + 6}{(n - 2)^2} ,
	\end{align*}
	
	\begin{align*}
		&\Rightarrow (n - 2)^3 w^3 \mathbb F\\
		=& (n - 1)(n - 2)^3 w^3 q^3 + w\Big[ (n - 1)(n - 2)w  -  2n^2 + 3n \Big] (n - 2)^2w^2q^2 \\
		&+ (n - 1)w^2\Big[ -(n - 1)(n - 2)w + n^2 + 3n - 9 \Big](n - 2)wq  \\
		&+  2(n - 1) (n - 2)^2w^4 + (-2n^2 + 6)(n - 2) w^3\\
		=& (n - 1)\Big[ - (n - 2)w^2  + (n - 1)w +  1\Big]^3\\
		& + w\Big[ (n - 1)(n - 2)w  -  2n^2 + 3n \Big] \Big[ - (n - 2)w^2  + (n - 1)w +  1\Big]^2 \\
		&+ (n - 1)w^2\Big[ -(n - 1)(n - 2)w + n^2 + 3n - 9 \Big]\Big[ - (n - 2)w^2  + (n - 1)w +  1\Big] \\
		& +  2(n - 1) (n - 2)^2w^4 + (-2n^2 + 6)(n - 2) w^3\\
		=& \Big[    (-n^2 + n + 1)w +  n - 1  \Big]\Big[ - (n - 2)w^2  + (n - 1)w +  1\Big]^2 \\
		&+ (n - 1)w^2\Big[ -(n - 1)(n - 2)w + n^2 + 3n - 9 \Big]\Big[ - (n - 2)w^2  + (n - 1)w +  1\Big] \\
		& +  2(n - 1) (n - 2)^2w^4 + (-2n^2 + 6)(n - 2) w^3\\
		=& \Bigg[   (n - 2)^2w^3  +  3(n - 2)(n - 1)w^2  - (n - 2)w   + n - 1 \Bigg]\Big[ - (n - 2)w^2  + (n - 1)w +  1\Big] \\
		& +  2(n - 1) (n - 2)^2w^4 + (-2n^2 + 6)(n - 2) w^3\\
		=& -(n - 2)^3 w^5 - 3(n - 2)^2(n - 1)w^4 + (n - 2)^2 w^3 - (n - 1)(n - 2)w^2\\
		&+ (n - 1)(n - 2)^2w^4 + 3(n - 2)(n - 1)^2w^3 - (n - 1)(n - 2)w^2 + (n - 1)^2 w\\
		&+   (n - 2)^2w^3  +  3(n - 2)(n - 1)w^2  - (n - 2)w   + n - 1  \\
		&+  2(n - 1) (n - 2)^2w^4 + (-2n^2 + 6)(n - 2) w^3\\
		=& -(n - 2)^3 w^5  + (n - 2)(n^2 - 4n + 5) w^3 + (n - 1)(n - 2)w^2  + (n^2 - 3n + 3) w   + n - 1  \\
		=&(w + 1) \Bigg[ -(n - 2)^3 w^4 + (n - 2)^3 w^3 + (n - 2)w^2 + (n - 2)^2  w + n - 1 \Bigg]\\
		=&(w + 1) \Big[ -(n - 2) w + (n - 1)\Big]\cdot \Big[ (n - 2)w^2 + 1\Big] \cdot \Big[ (n - 2)w + 1 \Big].
	\end{align*}
	So we deduce that when $0 < w < \frac{n - 1}{n - 2}$,
	\begin{equation}
		\mathbb F > 0.
	\end{equation}

%
%
	\end{proof}

Now let us recall that
	\begin{equation*}
		\begin{aligned}
			G(p)&:= \frac{1}{p_1} + \frac{1}{q_1} \\
			&= \frac{2(q - 2)}{(\gamma + 4)(\gamma + 2q)}B(p) + \frac{\gamma + 2}{\gamma + 2q}\\
			&= \frac{2(q - 2)}{\gamma + 2q}\left[\frac{a}{c} + \frac{-(\gamma + 2q)(\alpha_2 + \gamma + 2)}{c^2p + cd}\right] + \frac{\gamma + 2}{\gamma + 2q}\\
			&= \frac{2(2 - q)(\alpha_2 + \gamma + 2)}{c^2p + cd} + \frac{\gamma + 2}{\gamma + 4}.
		\end{aligned}
	\end{equation*}
	
	\begin{claim}
		When $p > 1 - q$ and $\frac{1}{n - 1} < q < 2$, we have $cp + d > 0$.
	\end{claim}
	
	\begin{proof}
		\begin{equation*}
			\begin{aligned}
				&cp + d\\
				&= (\gamma + 4)p + (2 - q)\alpha_2 + \gamma + 2q\\
				&> (\gamma + 4)(1 - q) + (2 - q)\alpha_2 + \gamma + 2q\\
				&= (2 - q)(\alpha_2 + \gamma + 2) > 0.
			\end{aligned}
		\end{equation*}
	\end{proof}
	Since $cp + d > 0, q < 2, c = \gamma + 4 > 0, \alpha_2 + \gamma + 2 > 0$, we deduce that $G(p)$ is decreasing with respect to $p$.
	\begin{claim}
		$p_1, q_1 > 0$.
	\end{claim}
	\begin{proof}
		We only need to show that 
		\begin{numcases}{}
			\frac{\gamma + 2q}{\gamma + 2 - B} > 0,\\
			B > 0.
		\end{numcases}
	In fact, we have 
	\begin{align*}
		\frac{\gamma + 2q}{\gamma + 2 - B} &= \frac{\gamma + 2q}{- \frac{bc - ad}{cp + d}}\\
		&= \frac{cp + d}{\alpha_2 + \gamma + 2} > 0.
	\end{align*}

	\begin{itemize}
		\item If $q \leq 1$, then
	\begin{align*}
		ap + b &:= (\gamma + 2)p + (1 - q)\alpha_2\\
		&= (n - 2)q p + (1 - q)\alpha_2\\
		&> (n - 2)q(1 - q) + (1 - q)\alpha_2\\
		&= (1 - q)(\alpha_2 + \gamma + 2) > 0.
	\end{align*}
	
	\item If $1 < q < 2$ and $n \geq 4$, then
	\begin{align*}
		&\alpha_2 < 0,\\
		\Leftrightarrow&  -(n - 1)p_2q - 2\frac{n - 1}{n - 2}q + 2p_2 + \frac{2}{n - 2} + (n - 2)p_2q^2 < 0,\\
		\Leftrightarrow & \Big[ (n - 2)q^2 - (n - 1)q + 2\Big]p_2 - \frac{2n - 2}{n - 2}q + \frac{2}{n - 2} < 0,\\
		\Leftrightarrow& p_2 < \frac{(2n - 2)q - 2}{(n - 2)\Big[ (n - 2)q^2 - (n - 1)q + 2\Big]}.
	\end{align*}
	Recall that
	\begin{equation*}
		0 = p_2^2 + \Big[\frac{n - 1}{n - 2}q - \frac{n^2 - 3}{(n - 2)^2}\Big]p_2 + \frac{1 - (n - 1)q}{(n - 2)^2}.
	\end{equation*}
	We need to show that 
	\begin{align*}
		0 &< \left[  \frac{(2n - 2)q - 2}{(n - 2)\Big[ (n - 2)q^2 - (n - 1)q + 2\Big]}\right]^2 \\
		&+ \Big[\frac{n - 1}{n - 2}q - \frac{n^2 - 3}{(n - 2)^2}\Big]\left[  \frac{(2n - 2)q - 2}{(n - 2)\Big[ (n - 2)q^2 - (n - 1)q + 2\Big]}\right] + \frac{1 - (n - 1)q}{(n - 2)^2}.
	\end{align*}
	\begin{claim}\label{lemma_q}
		The above condition holds for any $n \geq 4, 1 < q < 2$.
	\end{claim}
	 So
	\begin{align*}
		ap + b &:= (\gamma + 2)p + (1 - q)\alpha_2\\
		&> (1 - q)\alpha_2 > 0.
	\end{align*}
	\end{itemize}
	
	\end{proof}
	
	\begin{proof}[Proof of claim \ref{lemma_q}]
		Firstly we can rewrite the condition as
		\begin{align*}
			0 &< \left[  \frac{(2n - 2)q - 2}{ (n - 2)q^2 - (n - 1)q + 2}\right]^2 \\
		&+ \Big[(n - 1)q - \frac{n^2 - 3}{n - 2}\Big]  \frac{(2n - 2)q - 2}{ (n - 2)q^2 - (n - 1)q + 2}+ 1 - (n - 1)q,\\
		\Leftrightarrow 0&< \Big[ (2n - 2)q - 2 \Big]^2 + \Big[(n - 1)q - \frac{n^2 - 3}{n - 2}\Big]\cdot \Big[(2n - 2)q - 2\Big]\Big[ (n - 2)q^2 - (n - 1)q + 2\Big]\\
		&+ \Big[ (n - 2)q^2 - (n - 1)q + 2\Big]^2\cdot \Big[1 - (n - 1)q \Big],\\
		\Leftrightarrow 0&< 4\Big[ (n - 1)q - 1 \Big] + 2\Big[(n - 1)q - \frac{n^2 - 3}{n - 2}\Big]\cdot \Big[ (n - 2)q^2 - (n - 1)q + 2\Big]\\
		&- \Big[ (n - 2)q^2 - (n - 1)q + 2\Big]^2 =: J(q)
		\end{align*}
		Note that $J(2) \equiv 0$ and
		\begin{align*}
			J(1) &= 4(n - 2) + 2\Big[n - 1 - \frac{n^2 - 3}{n - 2}\Big] - 1 \\
			&= \frac{4(n - 1.75)(n - 4)}{n - 2} \geq 0,
		\end{align*}
		when $n \geq 4$. 
		Define $t = q - 2$, then
		\begin{align*}
			J(q) &= 4\Big[ (n - 1)q - 1 \Big] + 2\Big[(n - 1)q - \frac{n^2 - 3}{n - 2}\Big]\cdot \Big[ (n - 2)q^2 - (n - 1)q + 2\Big]\\
		&- \Big[ (n - 2)q^2 - (n - 1)q + 2\Big]^2\\
		&= 4\Big[ (n - 1)q - 1 \Big] + \Big[-(n - 2)q^2 + 3(n - 1)q - 2 - \frac{2n^2 - 6}{n - 2}\Big]\cdot \Big[ (n - 2)q^2 - (n - 1)q + 2\Big]\\
		&= 4\Big[ (n - 1)(t + 2) - 1 \Big]\\
		& + \Big[-(n - 2)(t + 2)^2 + 3(n - 1)(t + 2) - 2 - \frac{2n^2 - 6}{n - 2}\Big]\cdot \Big[ (n - 2)(t + 2)^2 - (n - 1)(t + 2) + 2\Big]\\
		&= 4(n - 1)t + 4(2n - 3)\\
		&+ \Big[-(n - 2)t^2 - (n - 5)t - \frac{4n - 6}{n - 2}\Big]\cdot \Big[ (n - 2)t^2 + (3n - 7)t + 2n - 4 \Big]\\
		&= -(n - 2)^2t^4 - (4n^2 - 20n + 24)t^3 - (5n^2 - 26n + 37)t^2 - (n - 5)(2n - 4)t \\
		&- \frac{(4n - 6)(3n - 7)}{n - 2}t + 4(n - 1)t.
		\end{align*}
		Define 
		\begin{align*}
			K(t) &= \frac{1}{q - 2}J(q)\\
			&= -(n - 2)^2t^3 - (4n^2 - 20n + 24)t^2 - (5n^2 - 26n + 37)t - \frac{2(n - 1)^2(n - 3)}{n - 2}.
		\end{align*}
		Since
		\begin{align*}
			K'(t) =  -3(n - 2)^2t^2 - 2(4n^2 - 20n + 24)t - (5n^2 - 26n + 37),
		\end{align*}
		and 
		\begin{align*}
			\Delta = 4(4n^2 - 20n + 24)^2 - 12(n - 2)^2(5n^2 - 26n + 37),
		\end{align*}
		we know there are two cases:
		\begin{itemize}
			\item If $n = 4, 5, 6, 7$, then $\Delta < 0$ and $K(t)$ is decreasing. So for any $-1 < t < 0$, we have
			\begin{align*}
				K(t) < K(-1) < 0.
			\end{align*}
			\item If $n \geq 8$, then $K'(t)$ is decreasing when $-1 < t < 0$. So
			\begin{align*}
				K'(t) < K'(-1) = -2n - 1 < 0.
			\end{align*}
			Then $K(t)$ is decreasing.
		\end{itemize}

		
	\end{proof}

	\begin{claim}
		$G(1 - q) = 1$.
	\end{claim}

	Using $\gamma = (n - 2)q - 2$, we know that when $q \geq 1$,
	
	\begin{equation*}
		\begin{aligned}
			&n - \gamma - 4 = (n - 2)(1 - q) < 0,
		\end{aligned}
	\end{equation*}
	Then since 
	\begin{align*}
		&G(p) > 1 - \frac{2}{n},\\
				\Leftrightarrow &\frac{2(2 - q)(\alpha_2 + \gamma + 2)}{(\gamma + 4)p + (2 - q)\alpha_2 + \gamma + 2q} + \gamma + 2 > \frac{n - 2}{n}(\gamma + 4),\\
				\Leftrightarrow &\frac{2(2 - q)(\alpha_2 + \gamma + 2)}{(\gamma + 4)p + (2 - q)\alpha_2 + \gamma + 2q} > \frac{2}{n}(-\gamma + n - 4) ,\\
	\end{align*}
	so for any $p > 0$, $G(p) > 1 - \frac{2}{n}$. As a result, we only need to consider the case $\frac{1}{n - 1} < q < 1$.
	
	\begin{claim}
		When $\frac{1}{n - 1} < q < 1$, we have $G(p_2) > G(p_*) > 1 - \frac{2}{n}$. 
	\end{claim}
	\begin{proof}
		\begin{equation*}
			\begin{aligned}
				&G(p) > 1 - \frac{2}{n},\\
				\Leftrightarrow &n(2 - q)(\alpha_2 + \gamma + 2) > -(\gamma - n + 4)\Big[ (\gamma + 4)p + (2 - q)\alpha_2 + \gamma + 2q\Big],\\
				\Leftrightarrow &(n - \gamma - 4)(\gamma + 4)p < n(2 - q)(\alpha_2 + \gamma + 2) - (n - \gamma - 4)\Big[ (2 - q)\alpha_2 + \gamma + 2q\Big].
			\end{aligned}
		\end{equation*}
		Using $\gamma = (n - 2)q - 2$, it follows that
		\begin{equation*}
			\begin{aligned}
				\Leftrightarrow &(n - 2)(1 - q)\Big[ (n - 2)q + 2\Big]p < \Big[n - 2 + (2 - q)\alpha_2\Big] \Big[ (n - 2)q + 2\Big],\\
				\Leftrightarrow & (n - 2)(1 - q)p < n - 2 + (2 - q)\alpha_2.
			\end{aligned}
		\end{equation*}
		
		In fact, we have
		\begin{equation}
			\begin{aligned}
				& n - 2 + (2 - q)\alpha_2\\
				&> n - 2 + (2 - q)(-\gamma - q)\\
				&= (n - 2)(1 - q)p_*.
			\end{aligned}
		\end{equation}

	\end{proof}

	Finally we give the proof the Lemma \ref{compare}:
	\begin{proof}[Proof of Lemma \ref{compare}]
		According to the definition of $\mathbb G(p, q)$ in \eqref{8.23G},
		\begin{equation}
			\begin{aligned}
				&\mathbb G(p, q) < 0,\\
				\Leftrightarrow &0 > p^2 + \frac{n(n - 1)q^2 - (n^2 +n - 1)q - n - 2}{(n - 1)^2q + n - 2}p - \frac{n q^2}{(n - 1)^2q + n - 2}. 
			\end{aligned}
		\end{equation}
		Suppose $p_M$ is the larger solution to the equation $\mathbb G(p, q)= 0$ with respect to $p$ for fixed $q$. Define $p_{min} = \max (1 - q, 0)$. 
		By \eqref{8.231125}, we only need to show that for any $p_{min} < p < p_M$, it holds that
		
		\begin{equation}\label{8.23.1131}
			\begin{aligned}
				& \mathbb H(p, q) < \mathbb  G(p, q) < 0,\\
			\Leftrightarrow	&\left[\frac{n(n - 1)q^2 - (n^2 +n - 1)q - n - 2}{(n - 1)^2q + n - 2} - \frac{n - 1}{n - 2}q + \frac{n^2 - 3}{(n - 2)^2}\right]p - \frac{n q^2}{(n - 1)^2q + n - 2}\\
				& - \frac{1 - (n - 1)q}{(n - 2)^2} > 0.
			\end{aligned}
		\end{equation}
		If $0 \leq q < 2$, then 
		\begin{align}
			&\frac{n(n - 1)q^2 - (n^2 +n - 1)q - n - 2}{(n - 1)^2q + n - 2} - \frac{n - 1}{n - 2}q + \frac{n^2 - 3}{(n - 2)^2} > 0,\\
			\Leftrightarrow	&\Big[n(n - 1)q^2 - (n^2 +n - 1)q - n - 2\Big](n - 2)^2 > \Big[ (n - 1)^2q + n - 2 \Big]\Big[(n - 1)(n - 2)q - n^2 + 3\Big],\\
				\Leftrightarrow &-(n - 1)(n - 2) q^2 + (4n^2 - 10n + 5)q + n - 2 > 0,
		\end{align}
		where the last inequality naturally hold when $0 \leq q < 2, n\geq 3$.
		As a result, in order to get \eqref{8.23.1131}, we only need to show that
		\begin{equation*}
			\begin{aligned}
				&\frac{n(n - 1)q^2 - (n^2 +n - 1)q - n - 2}{(n - 1)^2q + n - 2}p_{min} - \frac{nq^2}{(n - 1)^2q + n - 2}\\
				&> \Big[\frac{n - 1}{n - 2}q - \frac{n^2 - 3}{(n - 2)^2}\Big]p_{min} + \frac{1 - (n - 1)q}{(n - 2)^2}.
			\end{aligned}
		\end{equation*}
		In fact, when $1\leq q < 2$, we have
		\begin{equation*}
			\begin{aligned}
				&- \frac{n q^2}{(n - 1)^2q + n - 2} - \frac{1 - (n - 1)q}{(n - 2)^2}\\
				&= \frac{-n(n - 2)^2q^2 + (n - 1)^3 q^2 - (n - 1)^2q - n + 2 + (n - 2)(n - 1)q}{\Big[(n - 1)^2q + n - 2\Big](n - 2)^2}\\
				&= \frac{(n^2 - n - 1)q^2 - (n - 1)q - n + 2 }{\Big[(n - 1)^2q + n - 2\Big](n - 2)^2}\\
				&\geq \frac{(n^2 - n - 1) - (n - 1) - n + 2 }{\Big[(n - 1)^2q + n - 2\Big](n - 2)^2} > 0.
			\end{aligned}
		\end{equation*}
		Thus we only need to consider the case $\frac{1}{n - 1} < q < 1$. At this time, $p_{min} = 1 - q$.
		\begin{equation*}
			\begin{aligned}
				&\frac{n(n - 1)q^2 - (n^2 +n - 1)q - n - 2}{(n - 1)^2q + n - 2}(1 - q) - \frac{nq^2}{(n - 1)^2q + n - 2}\\
				&> \Big[\frac{n - 1}{n - 2}q - \frac{n^2 - 3}{(n - 2)^2}\Big](1 - q) + \frac{1 - (n - 1)q}{(n - 2)^2},\\
				\Leftrightarrow &(n - 2)^2 \Bigg\{\Big[ n(n - 1)q^2 - (n^2 +n - 1)q - n - 2\Big](1 - q) - nq^2 \Bigg\}\\
				&> \Big[  (n - 1)^2q + n - 2 \Big] \Bigg\{  \Big[(n - 1)(n - 2)q - n^2 + 3\Big](1 - q) + 1 - (n - 1)q\Bigg\}\\
			\end{aligned}
		\end{equation*}

		\begin{equation*}
			\begin{aligned}
				\Leftrightarrow & (n - 2)^2 \Bigg\{ -n(n - 1)q^3 + (2n^2 - n - 1)q^2 - (n^2 - 3)q - n - 2 \Bigg\}\\
				&> \Big[  (n - 1)^2q + n - 2 \Big] \Bigg\{ -(n - 1)(n - 2)q^2 + (2n^2 - 4n)q - n^2 + 4 \Bigg\},\\
				\Leftrightarrow & (n - 1)(n - 2)q^3 -  4(n - 1)(n - 2)q^2 + 4(n - 1)(n - 2)q > 0.
			\end{aligned}
		\end{equation*}
		
	\end{proof}

\end{CJK}

\bibliography{MN241217}	
\bibliographystyle{plain}

\end{document}